\documentclass[12pt]{amsart}
\usepackage[hmargin=3cm,vmargin=3.5cm]{geometry}

\usepackage[latin1]{inputenc}
\usepackage{xspace,amssymb,amsfonts,euscript}
\usepackage{amsthm,amsmath,amscd,stmaryrd,latexsym}
\usepackage{palatino}
\usepackage{euscript}
\usepackage[numbers]{natbib}
\usepackage[rgb]{xcolor}

\usepackage{graphicx}
\usepackage[all]{xy}
\SelectTips{cm}{}

\usepackage{tikz}
\usetikzlibrary{calc}

\usepackage[dvips]{epsfig}
\usepackage{pinlabel}
\usepackage{psfrag,epstopdf}

\newcommand{\ig}[2]{\vcenter{\xy (0,0)*{\includegraphics[scale=#1]{fig/#2}} \endxy}}

\newcommand{\igc}[2]{\begin{center} \includegraphics[scale=#1]{fig/#2} \end{center}}

\IfFileExists{comments.sty}{\usepackage{comments}}

\newtheorem{thm}{Theorem}[section]
\newtheorem{lemma}[thm]{Lemma}

\newtheorem{prop}[thm]{Proposition}
\newtheorem{cor}[thm]{Corollary}
\newtheorem{claim}[thm]{Claim}

\newtheorem*{prop*}{Proposition}

\theoremstyle{definition}
\newtheorem{defn}[thm]{Definition}

\newtheorem{example}[thm]{Example}

\newtheorem{exercise}[thm]{Exercise}

\theoremstyle{remark}
\newtheorem{remark}[thm]{Remark}
\newtheorem{question}[thm]{Question}

\numberwithin{equation}{section}

    \def\CM{{\mathbb{C}}}

  \def\hg{{\mathfrak h}}

  \def\mg{{\mathfrak m}}

    \def\QM{{\mathbb{Q}}}
    \def\RM{{\mathbb{R}}}

    \def\ZM{{\mathbb{Z}}}

    \def\BC{{\mathcal{B}}}
    
    \def\DC{{\mathcal{D}}}
  \def\eb{{\mathbf e}}

\def\HB{{\mathbf H}}

    \def\OC{{\mathcal{O}}}
    \def\PC{{\mathcal{P}}}

\def\a{\alpha}

\def\d{\delta}

\def\l{\lambda}

\let\phi=\varphi

\usepackage{bbm}

\def\Z{{\mathbbm Z}}

\def\1{\mathbbm{1}}

\newcommand{\ot}{\otimes}
\newcommand{\pa}{\partial}
\newcommand{\co}{\colon}

\renewcommand{\to}{\rightarrow}

\renewcommand{\sl}{\mathfrak{sl}}

\newcommand{\Kar}{\textbf{Kar}}
\newcommand{\Hom}{{\rm Hom}}

\newcommand{\End}{{\rm End}}

\newcommand{\Ext}{{\rm Ext}}

\newcommand{\op}{{\rm op}}

\newcommand{\STL}{S\mathcal{T\!L}}
\newcommand{\un}{\underline}
\newcommand{\Zvv}{\ZM[v,v^{-1}]}

\title{Indecomposable Soergel bimodules for universal Coxeter groups}
\author{Ben Elias and Nicolas Libedinsky}
 
\begin{document}

\begin{abstract} We produce an explicit recursive formula which computes the idempotent projecting to any indecomposable Soergel bimodule for a universal Coxeter system. This gives the exact set of primes for which the positive characteristic analogue of Soergel's conjecture holds. Along the way, we introduce the multicolored Temperley-Lieb algebra.
\end{abstract}

\maketitle

\setcounter{tocdepth}{1}
\tableofcontents

\section{Introduction} 


Most of this paper is devoted to the introduction and the elementary representation theory of a certain 2-category, the \emph{multi-colored Temperley-Lieb 2-category}, defined in
\S\ref{sec:nTL}. It is a straightforward generalization of the more familiar Temperley-Lieb category appearing in the representation theory of $\sl_2$, though beyond the most natural case
of two colors, it seems not to have appeared previously in the literature. In earlier work of the first author \cite{EDihedral}, a connection was made between the two-colored Temperley-Lieb
2-category and the category of Soergel bimodules for the infinite dihedral group. The representation theory of the Temperley-Lieb algebra was then used to prove facts about these Soergel
bimodules, leading to answers for certain basic questions in positive characteristic Kazhdan-Lusztig theory. In this paper, we make a connection between the multi-colored Temperley-Lieb
2-category and the category of Soergel bimodules for universal Coxeter groups, and prove the analogous results in Kazhdan-Lusztig theory. The remainder of the introduction will fill in the
details.

\subsection{Kazhdan-Lusztig theory in positive characteristic}

In the year 1979, Kazhdan and Lusztig (abbreviated ``KL") introduced their celebrated \emph{KL polynomials} for any Coxeter system \cite{KaLu1}. These polynomials, living as
coefficients in the Iwahori-Hecke algebra, have become a fundamental tool in representation theory, geometry, and combinatorics. However, they are also a fundamental mystery. Despite
countless papers exploring the combinatorics of KL polynomials, very little is known outside of specific cases. The only infinite families of Coxeter groups for which we have a complete
understanding of KL polynomials are the dihedral groups (a simple exercise) and the universal Coxeter groups (a result of Dyer \cite{DyerUniversal}). Recall that universal Coxeter groups
are groups generated by involutions, with no other relations.

As we will discuss shortly, KL polynomials encode multiplicities attached to certain important categories (representation-theoretic, geometric, or otherwise) defined in characteristic
zero. For crystallographic Coxeter groups, one can choose an integral form of these categories, in order to define their characteristic $p$ analogs. One can encode the new
multiplicities in so-called \emph{$p$-KL polynomials}, which depend strongly on the specific value of $p$, and eventually (for $p$ large) agree with the ordinary KL polynomials. Far less
is known about the $p$-KL polynomials, nor is there a known algorithm to compute them within the Hecke algebra, as there is for ordinary KL polynomials. The following question is
already of great interest.

\begin{question} \label{q1} Given a crystallographic Coxeter group $W$ and a prime $p$, for which $w \in W$ does there exist a $y \le w$ such that the $p$-KL polynomial $h^p_{y,w} \in
\ZM[v,v^{-1}]$ disagrees with the ordinary KL polynomial $h_{y,w}$? \end{question}

For example, recent work of Williamson \cite{WilCounterexample} has found an infinite family of such quadruples $(W,p,w,y)$ in type $A$, refuting a well-known conjecture about Lusztig's
character formula. Answering this question for Weyl groups and affine Weyl groups would have significant import for modular representation theory (see \cite{Soe4}).

In order to make sense of the $p$-KL polynomial as encoding multiplicities, we must specify in which category we work. The name $p$-KL polynomial originally referred to
multiplicities in the category of \emph{parity sheaves} \cite{JMW}. Ultimately, we will use the diagrammatic category $\DC$ defined by the first author and Williamson in \cite{EWGR4SB},
though there are connections between $\DC$, parity sheaves, and the category of Soergel bimodules $\BC$ introduced by Soergel \cite{Soe5}. We shall motivate these connections in the
introduction, and work in \S\ref{sec:universal} entirely with $\DC$.

Question \ref{q1} then becomes a question about how ``big" the indecomposable objects in $\DC$ are. For the baby case of universal Coxeter groups, we construct the indecomposable objects
explicitly in the generic case, as the images of certain idempotents (already a new result in characteristic zero). By identifying the denominators in these idempotents, we determine which
finite characteristics will deviate from the generic behavior, thus answering Question \ref{q1}. Similar results for the other baby case, dihedral groups, can be easily extrapolated from
the first author's work \cite{EDihedral}. This is a small step along a very long and difficult road.

\subsection{Soergel bimodules and Soergel diagrammatics}

In 1992, Soergel \cite{Soe2} introduced an additive category $\BC = \BC(W,S,V,\Bbbk)$ of graded bimodules over a polynomial ring, whose objects have come to be known as \emph{Soergel
bimodules}. This category depends on a Coxeter system $(W,S)$, a field $\Bbbk$, and a finite dimensional representation $V$ of $W$ over $\Bbbk$ (see \cite{Soe5} for this general
definition). One important special case will be when $\Bbbk = \RM$ or $\CM$ and $V = V_{\mathrm{rootic}}$ is the \emph{rootic representation}\footnote{The rootic representation is called
the \emph{geometric representation} in Bourbaki \cite{bourbaki1968elements} or Humphreys \cite{Humphreys}. The notational preference is explained in \cite{LibLIF}.} of $(W,S)$. When $W$ is
crystallographic, its rootic representation can be defined over $\ZM$, and thus over any field.

The motivation for introducing $\BC$ is that, when $W$ is a Weyl group and $V$ its rootic representation in characteristic zero, there is an isomorphism between (a simplified version of)
$\BC$ and additive subcategories of the representation-theoretic and geometric categories which KL polynomials study. More precisely, this simplified version is equivalent to the
projective objects in the principal block of the BGG category $\OC$, or the semisimple $N$-equivariant perverse sheaves on the flag variety $G/B$. Here $G$ is a connected reductive complex algebraic group, $B$ denotes a Borel subgroup of $G$, and $N$ denotes the unipotent radical of $B$. One advantage of Soergel's approach is
that $\BC$ can be defined for any Coxeter group, even non-crystallographic groups for which there is no corresponding geometry or representation theory. Another advantage is that $\BC$
has a simple algebraic definition, allowing one to study KL theory using low-tech methods.

For our purposes, the relevant feature of $\BC$ (and, in a sense, of the other categories as well) is \emph{Soergel's categorification theorem} (\cite{Soe2},\cite{Soe5}),
which states that Soergel bimodules are a categorification of the Hecke algebra $\HB(W)$ of $W$. In other words, there is an isomorphism of $\Z[v,v^{-1}]$-algebras \begin{equation} \label{eq:catfn} \mathrm{ch}: [ \BC(W,S,V,\Bbbk) ] \longrightarrow \HB(W)\end{equation} from the split
Grothendieck group to the Hecke algebra. Soergel proved this result for any Coxeter group, and for any
representation $V$ which is \emph{reflection vector faithful} (see \cite{Soe5} for the definition) over a infinite field $\Bbbk$ of characteristic $\ne 2$.\footnote{Unfortunately,
$V_{\mathrm{rootic}}$ is not always reflection vector faithful. However, when $\Bbbk=\RM$, Soergel constructed an explicit reflection vector faithful representation analogous to
$V_{\mathrm{rootic}}$, and the second author \cite{LibRR} has given numerous theorems relating results for $V_{\mathrm{rootic}}$ to results for this explicit representation.} The
indecomposable objects $\{B_x\}_{x \in W}$ in $\BC$ are classified by elements of $W$, and they descend to some \emph{positive} basis $\{\mathrm{ch}[B_x]\}$ of the Hecke algebra (i.e. certain
coefficients are positive). The Hecke algebra possesses a natural (positive) basis, the KL basis $\{b_x\}_{w \in W}$ (encoded by the KL polynomials). This raises the following question.

\begin{question} \label{q2} Given $(W,S,V,\Bbbk)$ with $V$ reflection vector faithful, for which $w \in W$ will it be the case that $\mathrm{ch}[B_w] = b_w$? \end{question}

When $V$ is reflection vector faithful and $\Bbbk$ has characteristic $0$, it was conjectured by Soergel that every $w \in W$ has this property. When $W$ is a Weyl group, this conjecture
is equivalent to the famed Kazhdan-Lusztig conjecture, proven by Brylinski-Kashiwara \cite{brylinski1981kazhdan} and Beilinson-Bernstein \cite{beilinson1981localisation} using difficult
geometric techniques. Soergel hoped that the algebraic setting of $\BC$ would allow for a simpler solution. The baby case of the dihedral group was proven by Soergel \cite{Soe2}. The
universal Coxeter group case was done by Fiebig in \cite{fiebig2008combinatorics}, using the tool of moment graphs. It was also shown later by the second author in unpublished work, by
constructing idempotents using singular Soergel bimodules \cite{WilSSB}. The general case was recently proven by the first author and Williamson \cite{EWHodge} for $V_{\mathrm{rootic}}$
when $\Bbbk = \RM$; therefore, the KL basis $b_w$ really does encode something about characteristic 0 Soergel bimodules.

Soergel's categorification theorem implies that $\BC$ is the correct object to study in finite odd characteristic, so long as $V$ is reflection vector faithful. When $W$ is
a Weyl group, $V_{\textrm{rootic}}$ will be reflection vector faithful in characteristic $\ne 2$, and Question \ref{q1} is equivalent to Question \ref{q2}. However, an infinite
Coxeter group does not possess a faithful representation in positive characteristic, so that $\BC$ is not quite the correct category to study.

It is somewhat naive to assume that using the same definitions in characteristic $p$ will yield a category with similar properties. The most appropriate way to define a finite
characteristic analog of an additive (resp. abelian) $\RM$-linear category is to first choose an integral form. This involves finding a projective generator $P$ and a $\ZM$-algebra $E$
such that $E \ot_{\ZM} \RM \cong \End(P)$. Then one considers the category of projective (resp. all) $E \ot_{\ZM} \Bbbk$-modules for other fields $\Bbbk$. A typical choice for a generator
would be the sum $P_{\textrm{min}} = \oplus_{w \in W} B_w$ of all the indecomposable objects, but this choice makes computing $E$ quite difficult. In fact, the crux of Soergel's
construction is that $\BC$ (and category $\OC$ and perverse sheaves) has a nice combinatorial generator $P_{\textrm{BS}}$, the sum of all the \emph{Bott-Samelson objects}. Soergel
bimodules are, by definition, summands of Bott-Samelson bimodules, which in turn admit a simple description. In \cite{EWGR4SB}, the first author and Williamson show that the endomorphism
algebra $\End(P_{\textrm{BS}})$ also admits a nice combinatorial description, using so-called \emph{Soergel diagrams}.

In \cite{EWGR4SB} one defines a diagrammatic category $\DC$ depending on a \emph{realization}, which is roughly the data of $(W,S,V,\Bbbk)$ together with a choice of simple roots and
coroots (although $\Bbbk$ can be any commutative ring). There is an equivalence $\DC \cong \BC$ of monoidal categories when the latter is ``well behaved", i.e. when \eqref{eq:catfn} gives
an isomorphism and the indecomposable objects are parametrized by $W$. Under some minimal assumptions, $\DC$ is well behaved in this sense even when $\BC$ is not (such as when the
representation is not reflection vector faithful), justifying the statement that $\DC$ is the appropriate replacement for $\BC$. We still denote the indecomposable objects of $\DC$ by
$B_w$ for $w \in W$.

\begin{question} \label{q3} Given a realization over a complete local ring $\Bbbk$ where $[\DC] \cong \HB(W)$, for which $w \in W$ will it be the case that $\mathrm{ch}[B_w] = b_w$? \end{question}

This is the most general alternative to Question \ref{q1}, and it depends on the realization itself, not just on the characteristic of $\Bbbk$. When $W$ is crystallographic and the
realization is rootic in finite characteristic, $\DC$ agrees with the category of parity sheaves \cite{JMW} on the flag variety. Parity sheaves are a finite characteristic analog of
perverse sheaves\footnote{Perverse sheaves do exist in finite characteristic, but like Soergel bimodules for non-reflection-faithful representations, they do not possess the desired
categorification-related properties. Parity sheaves do.}, whose multiplicities were originally called $p$-KL polynomials. Therefore, Question \ref{q1} is a special case of Question
\ref{q3}.

Having chosen $P_{\textrm{BS}}$ rather than $P_{\textrm{min}}$ as our generator, one has an implicit definition of the indecomposable objects as certain summands of Bott-Samelson objects.
Finding an explicit construction of each $B_w$ is the true goal underlying all these Grothendieck-group-theoretic questions about their sizes. There is a general computational algorithm for
the idempotents (inside an endomorphism ring of a Bott-Samelson object) which project to each indecomposable summand (see \cite{LibLIF}), but this algorithm is unsatisfactory in that it
provides no insight into the dependence of this idempotent on the choice of realization. Instead, one hopes for an explicit formula (possibly inductive) for the idempotents in the generic
case, yielding explicit knowledge of their denominators. This is what we achieve for universal Coxeter groups. The case of a general Coxeter group seems to be drastically more difficult.

\begin{example} One can consider the most standard realization of a universal Coxeter group, arising from a symmetric Cartan matrix where each diagonal entry is $2$ and each non-diagonal
entry is $-2$. This realization can be defined over $\ZM$, and thus over any field $\Bbbk$, though only the characteristic of $\Bbbk$ is relevant here. Each element of the Coxeter group has
a unique reduced expression, having the form $w = s_1 s_2 s_3 \cdots s_d$ where $s_i$ are simple reflections and $s_i \ne s_{i+1}$. A maximal alternating subsequence of this reduced
expression is a consecutive subsequence $s_i s_{i+1} \cdots s_{i+k}$ (having length $k+1$), satisfying $s_i = s_{i+2} = \cdots$ and $s_{i+1} = s_{i+3} = \cdots$, and which is not contained
in a larger alternating consecutive subsequence. Our results (c.f. Proposition \ref{prop:JWdescriptive}) state that $\mathrm{ch}[B_w] = b_w$ if and only if the binomial coefficients $\binom{k}{m}$ are
invertible in $\Bbbk$ for each $0 \le m \le k$, for every $k$ such that $k+1$ is the length of a maximal alternating subsequence of $w$. \end{example}

\subsection{Techniques}

In \cite{EDihedral}, the first author demonstrated that Soergel bimodules for the infinite dihedral group were intimately related to the Temperley-Lieb algebra which arises in
$\mathfrak{sl}_2$ representation theory. The familiar Jones-Wenzl idempotents in the Temperley-Lieb algebra were transformed into idempotent endomorphisms of Bott-Samelson bimodules,
projecting to the indecomposable summands. This paper takes these ideas to their natural conclusion, producing a relationship between the multicolored Temperley-Lieb $2$-category and
Bott-Samelson bimodules (or rather, their diagrammatic analogues) for the corresponding universal Coxeter group.

In \S\ref{sec:nTL} we define the multicolored Temperley-Lieb $2$-category and explore its representation theory. We define the analogues of Jones-Wenzl idempotents. We provide a recursive
formula for these idempotents, allowing one to categorify Dyer's inductive formula for the KL basis. We also provide an immediate formula for these idempotents in terms of the Jones-Wenzl
idempotents in the usual Temperley-Lieb algebra (which unfortunately have no easy closed formula, though see \cite{Mor}). This second formula implies a criterion for when the idempotent is
not defined, which will lead to the answer to Question \ref{q3}. Specifically, $\mathrm{ch}[B_w]=b_w$ so long as certain ``colored quantum binomial coefficients" are invertible.

This answer to Question \ref{q3} relies on the fact that Jones-Wenzl projectors in the usual Temperley-Lieb algebra exist if and only if certain quantum binomial coefficients are
invertible. Though fundamental to the theory of Temperley-Lieb algebras, this fact does not seem to appear in the literature. To remedy this, we have included an appendix written by Ben
Webster, with a proof of the result over a general ring (see Theorem \ref{thm:JW-exist}).

In \S\ref{sec:universal} we define the diagrammatic category $\DC$ associated to the most general realization of a universal Coxeter group. Our definition is purely diagrammatic, using
the results of \cite{EWGR4SB}, and thus we never mention Soergel bimodules. We prove the main theorem: that the multicolored Temperley-Lieb $2$-category encodes all the morphisms of
minimal degree in $\DC$. Therefore, the Jones-Wenzl analogues provide all the indecomposable idempotents in $\DC$.

\section{The $n$-colored Temperley-Lieb 2-category}
\label{sec:nTL}

\subsection{Definitions} \label{subsec:definitions}

We assume that the reader is familiar with several topics, for which we give some references. Introductory material on the Temperley-Lieb category can be found in
\cite{WestburyTL,GooWen}. An introduction to (strict) 2-categories and their diagrammatic presentations can be found in \cite[section 2]{LauDiagrams}. An introduction to
Karoubi envelopes (of categories and 2-categories) can be found in \cite{LauSL2}.

Let $S$ be a finite set with size $n$. We associate a color to each element of $S$, blue to $b$ and red to $r$, etcetera. Let $\d$ be an indeterminate.

\begin{defn} \label{defn:STL} The \emph{$S$-colored} or \emph{$n$-colored Temperley-Lieb 2-category} $\STL$ is the $\ZM[\d]$-linear 2-category with objects $S$, having the following
presentation. There is a generating 1-morphism from $b$ to $r$, for each pair of distinct elements $b \ne r \in S$. Therefore, a general 1-morphism can be represented uniquely by the
(non-empty) sequence $\un{x} = s_1 s_2 \ldots s_m$ of colors through which it passes, satisfying $s_i \ne s_{i+1}$ for all $i$. We read 1-morphisms from right to left, so that $\un{x}$ has
source $s_m$ and target $s_1$. We say the 1-morphism has \emph{length} $\ell(\un{x}) = m$ (this is not additive under composition; it would be additive if we set $\ell(\un{x})=m-1$
instead). For instance, the identity 1-morphism of any object $s \in S$ has length $1$. We represent a composition of 1-morphisms diagrammatically as a sequence of dots on the line,
separating regions of different colors.

\begin{example} The $1$-morphism $brgryb$: $\ig{1}{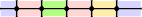}$. \end{example}

The 2-morphisms are generated by colored cups and caps. More precisely, for each $b \in S$ and for each $r \in S \setminus b$ there is a cap map $brb \to b$ and a cup map $b \to brb$,
as pictured below. \igc{1}{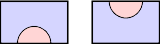} In these diagrams, morphisms are read from bottom to top.

There are two types of relations, which hold for every possible coloring of regions.

\begin{equation} \label{eq:isotopy} \ig{1}{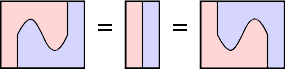} \end{equation}

\begin{equation} \label{eq:circle} 	{
	\labellist
	\small\hair 2pt
	 \pinlabel {$-\d$} [ ] at 60 20
	\endlabellist
	\centering
	\ig{1}{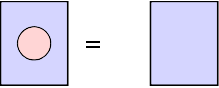}
	} \end{equation}

This ends the definition.
\end{defn}

\begin{remark} We shall actually be interested in a generalization of this definition, introduced in section \ref{subsec:generalizations}. For pedagogical reasons, however,
we shall temporarily work with this more familiar-looking definition. \end{remark}

Let $CM(m,k)$ denote the set of $(m,k)$-crossingless matchings in the planar strip (see \cite[section 1]{WestburyTL} for the definition). Given any element of $CM(m,k)$, one can color the
regions by elements of $S$ so that no two adjacent regions have the same color. The resulting diagram will represent some 2-morphism in $\STL$. Conversely, every 2-morphism in $\STL$ is a
$\ZM[\d]$-linear combination of such colored crossingless matchings.

\begin{defn} For fixed 1-morphisms $\un{x} = s_1 s_2 \ldots s_{m+1}$ and $\un{y} = t_1 t_2 \ldots t_{k+1}$, we let $CM(\un{x},\un{y})$ denote the subset of $(m,k)$-crossingless matchings
which can be consistently colored to yield a 2-morphism in $\Hom(\un{x},\un{y})$. \end{defn}

For example, $CM(\un{x},\un{y}) = \emptyset$ unless $s_1 = t_1$ and $s_{m+1} = t_{k+1}$.

\begin{example} An element of $CM(grgyrybgbyb,gyrorybrb)$: \igc{1}{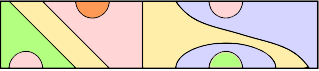} The reader can verify that only four elements of $CM(10,8)$ actually give rise to an element of $CM(grgyrybgbyb,gyrorybrb)$.  \end{example}

\begin{lemma} The set $CM(\un{x},\un{y})$ forms a $\ZM[\d]$-basis for $\Hom(\un{x},\un{y})$. \end{lemma}

\begin{proof} This can be proven in exactly the same way that one proves that $CM(m,k)$ is a $\ZM[\d]$-basis for $\Hom(m,k)$ in the usual Temperley-Lieb category. Here is a sketch of such a
proof.

The generators of $\STL$ allow one to construct a morphism in $\Hom(\un{x}, \un{y})$ for each planar 1-manifold with boundary, with appropriate coloring on the boundary. Relation
\eqref{eq:isotopy} implies that any two isotypic 1-manifolds are equal. Relation \eqref{eq:circle} is equivalent to a family of relations which states that every planar 1-manifold is equal
to $(-\d)^k$ times the underlying crossingless matching, where $k$ is the number of loops removed. It is clear that this family of relations can not create any linear dependencies between
crossingless matchings. \end{proof}

\begin{example} \label{example:identity} Note that $CM(\un{x},\un{x})$ always contains the identity crossingless matching $\1$, but may contain no others. For instance, if $s_i \ne
s_{i+2}$ for all $i$, then there can be no cups or caps, and therefore $CM(\un{x},\un{x})$ only contains the identity. \end{example}

\begin{example} When $n=1$ and $S=\{b\}$, the 2-category $\STL$ is very boring, having a unique 1-morphism $b$ with $\End(b)=\ZM[\d]$. \end{example}

\begin{example} \label{example:n=2} When $n=2$ and $S=\{r,b\}$, the 1-morphisms are alternating sequences $\un{x}=rbrb\ldots$. When $\un{x}$ and $\un{y}$ both begin with $r$, and have lengths $m+1$ and
$k+1$ respectively, then $CM(\un{x},\un{y}) = CM(m,k)$. In fact, there is an equivalence of categories between the usual Temperley-Lieb category and the full subcategory of $\STL$
obtained by considering only 1-morphisms beginning with $r$. In many senses, the two-colored Temperley-Lieb category is more natural than the usual Temperley-Lieb category, because the
representation theory of $\mathfrak{sl}_2$  is naturally $\ZM/2\ZM $ graded (even and odd representations), where we identify $\ZM/2\ZM $ with the quotient $\Lambda_{\textrm{wt}} / \Lambda_{\textrm{rt}}$ of the integral weight lattice by the root lattice. For more on this, see \cite{EQAGS}. \end{example}

\begin{exercise} Conversely, let $S$ be arbitrary, and suppose that $\un{x}$ and $\un{y}$ begin with $r$ and have lengths $m+1$ and $k+1$ repsectively. Then one has an equality
$CM(\un{x},\un{y}) = CM(m,k)$ if and only if both spaces are empty (i.e. $k+m$ is odd), or $\un{x}$ and $\un{y}$ both alternate between $r$ and another color $b$. \end{exercise}

By flipping diagrams upside-down, one obtains a bijection between $CM(\un{x},\un{y})$ and $CM(\un{y},\un{x})$. This extends to an antiinvolution $\iota$ on $\STL$.

\subsection{The Karoubi envelope} \label{subsec:karoubi}

In this paper, $\Bbbk$ will always be a commutative ring, perhaps with extra structure. In this section and the next, $\Bbbk$ will be a complete local $\ZM[\d]$-algebra. We now work in
the 2-category $\STL \ot_{\ZM[\d]} \Bbbk$ obtained by base change, and abusively denote this category $\STL$.

\begin{remark} It is well-known that the usual Temperley-Lieb category is cellular (see \cite[section 2]{WestburyCellular} for the definition). In fact, it is an especially nice kind of
cellular category known as an \emph{object-adapted cellular category}, meaning roughly that the cells correspond to some objects in the category, and that the top cell of each of these
objects contains only the identity map. The monoidal structure is usually ignored when studying the cellular structure (certainly the theory of monoidal cellular categories has not been
thoroughly explored).

Similarly, the 2-category $\STL$, when viewed as a 1-category by forgetting the structure of horizontal concatenation, is an (object-adapted) cellular category, using a direct adaptation
of the structure on the usual Temperley-Lieb category. The features of the Karoubi envelope of $\STL$ that we discuss below are in fact rather general properties of object-adapted
cellular categories, but we give complete proofs. In particular, references to ``shorter sequences" below should be replaced with references to the cellular partial order. Future work of
the first author will contain more discussion of object-adapted cellular categories. \end{remark}

Fix a 1-morphism $\un{x}$ of length $m+1$. A key property of the set $CM(m,m)$, which we used implicitly in Example \ref{example:identity}, is that every diagram except the identity
contains a cap on bottom and a cup on top. In particular, the span of the non-identity diagrams in $CM(\un{x},\un{x})$ forms a two-sided ideal $I_{< \un{x}} \subset \End(\un{x})$,
whose quotient is free of rank $1$ over $\Bbbk$, spanned by the identity.

Suppose that one can decompose $\1 \in \End(\un{x})$ into a sum $\1 = \sum e_i$ of orthogonal indecomposable idempotents. It is easy to see, by working modulo $I_{< \un{x}}$, that there
is a unique idempotent $e_0$ with a non-zero coefficient of the identity (in the basis $CM(\un{x},\un{x})$), and this coefficient is 1. The remaining idempotents lie within $I_{<
\un{x}}$. Our goal is to prove that within the Karoubi envelope $\Kar(\STL)$, the object $\un{x}$ has a unique indecomposable summand $V_{\un{x}}$ which is not a summand of $\un{y}$ for
any shorter sequence; it is the image of $e_0$. In other words, the idempotents within $I_{< \un{x}}$ actually factor through shorter sequences $\un{y}$.

\begin{lemma} Suppose that $\Bbbk$ is a complete local ring. Then $\Kar(\STL)$ has the Krull-Schmidt property. \end{lemma}

\begin{proof} This is a general fact for $\Bbbk$-linear categories with finite dimensional Hom spaces. A similar proof can be found in \cite[Proposition 1.1]{LenzingLectureNotes}.
\end{proof}

\begin{prop} Suppose that $\Bbbk$ is a complete local ring. For each sequence $\un{x}$ choose a decomposition $\1 = \sum e_i$ into orthogonal indecomposable idempotents, such that $e_0 =
\1$ modulo $I_{< \un{x}}$. Let $V_{\un{x}}$ denote the image of $e_0$, an object in $\Kar(\STL)$. Then the collection of all $V_{\un{x}}$ over all sequences $\un{x}$ form a complete list
of non-isomorphic indecomposable objects in $\Kar(\STL)$, and $$\un{x} \cong V_{\un{x}} \oplus \bigoplus_{\ell(\un{y}) < \ell(\un{x})} V_{\un{y}}^{\oplus m_{\un{y}}}.$$ In particular, by
the Krull-Schmidt property, the object $V_{\un{x}}$ is independent of the choice of idempotent decomposition, up to isomorphism. \end{prop}

The sections which follow will give a more intuitive and obvious proof under some additional assumptions, and the novice reader should skip there. We now provide a general proof, which is
adapted directly from the proof of the Soergel Categorification Theorem found in \cite[section 6.6]{EWGR4SB}.

\begin{proof} It is not hard to reduce to the following statement: for each $\un{x}$ and each indecomposable idempotent $e \in \End(\un{x})$, the corresponding object $V$ in $\Kar(\STL)$
is isomorphic to $V_{\un{y}}$ for some $\un{y}$ with $\ell(\un{y}) \le \ell(\un{x})$, with equality if and only if $e = \1$ modulo $I_{< \un{x}}$.

Any diagram in $CM(\un{x},\un{x})$ factors as $S \circ T$, for some triple $(\un{z},S,T)$ where $\ell(\un{z}) \le \ell(\un{x})$, $T \in \CM(\un{x},\un{z})$ is a cap diagram, and $S \in
\CM(\un{z},\un{x})$ is a cup diagram (see \cite[section 2]{WestburyTL} for the definition of cap and cup diagrams). Then one can expand $e$ in the diagram basis \[ e = \sum_{\un{z}}
\sum_{(\un{z},S,T)} a_{S,T} S \circ T \] with some coefficients $a_{S,T} \in \Bbbk$. Choose a sequence $\un{y}$ of maximal length such that there exists a triple $(\un{y},S,T)$ with
$a_{S,T} \ne 0$. Note that $\un{y} = \un{x}$ precisely when $e = \1$ modulo $I_{< \un{x}}$.

We wish to show that there is some triple $(\un{y},X,Y)$ such that \[ Y \circ e \circ X \in \Bbbk^\times \subset \Bbbk = \End(\un{y})/I_{<\un{y}}. \] Let $I_{\ngeq \un{y}}$ denote the
ideal of all morphisms which factor through any sequence shorter than $\un{y}$, or any other sequence of the same length. We now proceed to work in the quotient category $\STL / I_{\ngeq
\un{y}}$. The image of $e$ is still a nonzero idempotent, expanded as above except only using triples with $\un{z}=\un{y}$. Let $\mg$ denote the maximal ideal of $\Bbbk$. Suppose that \[T
\circ e \circ S \in \mg \subset \Bbbk = \End(\un{y})/I_{< \un{y}} \] for all triples $(\un{y},S,T)$. By expanding $e^3 = e$ one can deduce that each $a_{S,T} \in \mg$. But this is a
contradiction, as $\mg \End(\un{x})$ is contained in the Jacobson radical of $\End(\un{x})$, and no non-zero idempotent can be contained in the Jacobson radical.

The map $Y \circ e$ induces a map $V \to \un{y}$, and $e \circ X$ induces a map in the other direction. By composing these further with the chosen idempotent $e_0$ inside $\un{y}$, we
obtain maps $e_0 \circ Y \circ e \co V \to V_{\un{y}}$ and $e \circ X \circ e_0 \co V_{\un{y}} \to V$. Composing these maps we get an endomorphism of $V_{\un{y}}$ which projects to an
invertible map in $\End(V_{\un{y}}) / I_{< \un{y}} = \Bbbk$, so that it must be invertible in the local ring $\End(V_{\un{y}})$. Therefore, $V_{\un{y}}$ occurs as a summand of $V$, and
since $V$ is indecomposable, we have $V \cong V_{\un{y}}$. \end{proof}

\subsection{Orthogonality} \label{subsec:perpendicular}

In the rest of this chapter, we discuss the case when $e_0$ has an alternative description as the unique idempotent perpendicular to $I_{< \un{x}}$. In particular, $e_0$ is canonically
defined, and $V_{\un{x}}$ is well-defined up to unique isomorphism. In this case, the recursive formula of the following section will make the fact that all other idempotents factor
through shorter expressions immediately obvious.

Let $T=T(\un{x}) \subset \End(\un{x})$ be the right perpendicular space to $I_{< \un{x}}$. In other words, $T$ is the $\Bbbk$-module consisting of all $f \in \End(\un{x})$ such that
$cf=0$ for any cap $c$.

{ \labellist \small\hair 2pt \pinlabel {$f$} [ ] at 39 22 \endlabellist \centering \igc{1}{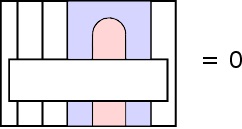} }

\noindent (The lack of color in some regions is supposed to represent the irrelevance of those colors.) Similarly, let $B=B(\un{x}) \subset \End(\un{x})$ be the left perpendicular space, the
$\Bbbk$-module consisting of all $f \in \End(\un{x})$ such that $fc=0$ for any cup $c$. Clearly $\iota(B)=T$.

\begin{claim} \label{claim:oneD} Suppose that $B(\un{x})$ contains an element $f$ for which the coefficient of the identity is invertible in $\Bbbk$. Then $B=T$ and both are spanned by
$f$. Every element of $B$ is fixed by $\iota$. Moreover, $B$ contains a unique idempotent $JW(\un{x})$, which is determined within $B$ by the fact that the coefficient of the identity
is $1$. The idempotent $JW(\un{x})$ is indecomposable and central, and it is the unique indecomposable idempotent not contained in $I_{< \un{x}}$. \end{claim}

\begin{proof} Let us write $f = \l \1 + f'$, where $f' \in I_{<\un{x}}$ and $\l \in \Bbbk$ is invertible. For any $g \in T$ one has $f' g=0$, since any non-identity diagram has a cap
on bottom. Therefore $fg = \l g$. By the same token, if $g \in \mu \1 + I_{<\un{x}}$ then $fg = \mu f$. In particular, $g = \mu \l^{-1} f$. This proves that every element of $T$ is in
the $\Bbbk$-span of $f$. By the same token, every element of $B$ is in the $\Bbbk$-span of $\iota(f)$. Thus $\iota(f) \in T$ is a multiple of $f$, and the coefficient of the
identity is also $\l$, so that $\iota(f)=f$. Therefore $B=T = \Bbbk \cdot f$, and every element is $\iota$-fixed. Moreover, letting $JW(\un{x}) = \l^{-1} f$, the above argument proves
that $JW(\un{x})^2 = JW(\un{x})$.

If $g \in \End(\un{x})$ is any element, and $g \in \mu \1 + I_{<\un{x}}$, then $JW(\un{x}) g = g JW(\un{x}) = \mu JW(\un{x})$, so that $JW(\un{x})$ is central. In particular, the ideal of
$JW(\un{x})$ is free of rank 1, from which it follows that the idempotent is indecomposable. Moreover, $JW(\un{x}) g = 0$ if and only if $g \in I_{< \un{x}}$, so that every other
indecomposable idempotent is contained in $I_{< \un{x}}$. \end{proof}

This idempotent $JW(\un{x})$, which we call the \emph{top idempotent}, is akin to the Jones-Wenzl projectors defined for usual Temperley-Lieb algebras. We draw $JW(\un{x})$ as a box
labeled by $\un{x}$, as in this example.

{
\labellist
\small\hair 2pt
 \pinlabel {$rgybrbg$} [ ] at 41 24
\endlabellist
\centering
\igc{1}{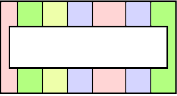}
}

It is possible that $B$ does not contain any element with invertible coefficient of the identity, in which case we say that the top idempotent does not exist. In this case, the special
idempotent $e_0$ discussed above can be more complicated, and will not be orthogonal to $I_{< \un{x}}$.

\subsection{A recursive formula for top idempotents} \label{subsec:recursive}

We use quantum number notation to indicate certain elements in $\Bbbk$. Let $[2] \in \Bbbk$ be the image of $\d$, and let $[1]=1$ and $[0]=0$. One defines the
quantum number $[m] \in \Bbbk$ for $m \in \ZM$ by the recursive formula \begin{equation} [2][m]=[m+1] + [m-1]. \end{equation}

Given a sequence of colors $\un{x}$, a \emph{subsequence} $\un{y} \subset \un{x}$ will always indicate a consecutive subsequence. A subsequence is \emph{alternating} if it alternates
between two colors in $S$; it is \emph{maximal alternating} if it can not be extended to a longer alternating subsequence. An \emph{initial subsequence} is a sequence consisting of the
first $k$ colors in $\un{x}$, and a \emph{final subsequence} is a sequence consisting of the last $k$ colors in $\un{x}$, for any $1 \le k \le \ell(x)$. The \emph{tail} of $\un{x}$ is the
maximal alternating final subsequence. For example, the tail of $gbryrgbrbrb$ is the final 5 colors $brbrb$.

In this section we provide a recursive formula for top idempotents under the assumption that certain quantum numbers are invertible in $\Bbbk$. This imitates a formula from \cite{Wenzl},
giving the Jones-Wenzl projectors in the usual Temperley-Lieb category.

\begin{prop} \label{prop:JWrecursive} When $\ell(\un{x}) \le 2$, the idempotent $JW(\un{x})$ exists, and is equal to the identity map. Now suppose that $\un{x} = \ldots r b$ and that
$JW(\un{y})$ exists for all initial subsequences of $\un{x}$. Extending $\un{x}$ by a color $g \ne r,b$ one has \begin{equation} \label{eq:JWeasy} 	{
	\labellist
	\small\hair 2pt
	 \pinlabel {$\un{x}g$} [ ] at 42 22
	 \pinlabel {$\un{x}$} [ ] at 145 22
	\endlabellist
	\centering
	\ig{1}{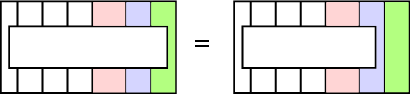}
	}. \end{equation}
Extending $\un{x}$ by $r$, when $[k]$ is invertible one has \begin{equation} 	{
	\labellist
	\small\hair 2pt
	 \pinlabel {$\un{x}r$} [ ] at 43 53
	 \pinlabel {$\un{x}$} [ ] at 147 53
	 \pinlabel {$\un{x}$} [ ] at 273 86
	 \pinlabel {$\un{x}$} [ ] at 273 21
	 \pinlabel {$\un{z}$} [ ] at 262 53
	 \pinlabel {$\frac{[k-1]}{[k]}$} [ ] at 221 57
	\endlabellist
	\centering
	\ig{1}{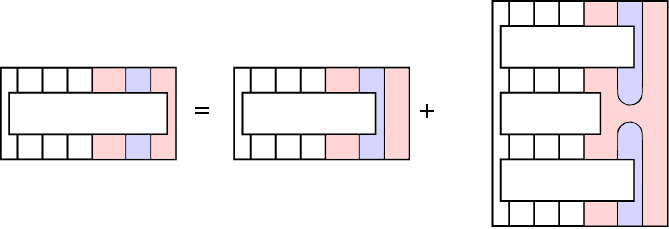}
	}. \label{eq:JWhard} \end{equation} The sequence $\un{z} = \ldots r$ is the initial subsequence of $\un{x}$ which is only missing the final
$b$. The number $k$ appearing in \eqref{eq:JWhard} is the length of the tail of $\un{x}$. Moreover, the map \begin{equation} \label{eq:loweridempotent} 	{
	\labellist
	\small\hair 2pt
	 \pinlabel {$\un{x}$} [ ] at 35 86
	 \pinlabel {$\un{x}$} [ ] at 35 21
	 \pinlabel {$\un{z}$} [ ] at 27 53
	 \pinlabel {$-\frac{[k-1]}{[k]}$} [ ] at -20 57
	\endlabellist
	\centering
	\ig{1}{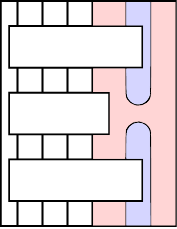}
	} \end{equation} is an idempotent in $\End(\un{x}r)$ orthogonal to $JW(\un{x}r)$. If $[k]$ is not invertible then $JW(\un{x}r)$ does not exist. \end{prop}

Note that the idempotent $JW(\un{z})$ which appears in \eqref{eq:loweridempotent} can be replaced by the identity map of $\un{z}$. After all, any non-identity term will have a cup
on top, which will annihilate $JW(\un{x})$. We included the idempotent $JW(\un{z})$ because it implies that the idempotent \eqref{eq:loweridempotent} factors through $V_{\un{z}}$
inside $\un{z}$.

\begin{exercise} When \eqref{eq:JWhard} holds, show that the coefficient of \begin{equation} \label{rightcoeff} \ig{1}{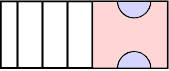} \end{equation} in $JW(\un{x}r)$ is precisely
$\frac{[k-1]}{[k]}$. Hint: replace each $JW(\un{x})$ with a linear combination of crossingless matchings in the RHS of \eqref{eq:JWhard}, and observe that only a single term could
possibly contribute to this coefficient. \end{exercise}

\begin{proof} When $\ell(\un{x}) \le 2$, $CM(\un{x},\un{x})$ only contains the identity map, and $I_{< \un{x}}=0$. It is clear that $B=T=\End(\un{x})$ and $JW(\un{x}) = \1_{\un{x}}$.
We assume henceforth that $\ell(\un{x}) \ge 2$, and that $JW(\un{y})$ exists for all initial subsequences $\un{y}$ of $\un{x}$. If $k$ is the length of the tail of $\un{x}$, then our
inductive hypothesis implies that $[l]$ is invertible for all $l<k$.

Suppose that $g \ne r,b$. Any non-identity diagram in $CM(\un{x}g, \un{x}g)$ must begin with a cup, and it is inconsistent with the coloring for this cup to involve the final strand.
Therefore any non-identity diagram will kill the RHS of \eqref{eq:JWeasy}, because a cup enters $JW(\un{x})$. The RHS of \eqref{eq:JWeasy} is clearly in $B(\un{x}g)$, and the
coefficient of the identity is equal to $1$ since this is true also in $JW(\un{x})$, so that the RHS is equal to $JW(\un{x}g)$.

Now we extend $\un{x}$ by $r$. We claim that \begin{equation}\label{eq:reverselower}{
\labellist
\small\hair 2pt
 \pinlabel {$\un{z}$} [ ] at 30 87
 \pinlabel {$\un{z}$} [ ] at 30 21
 \pinlabel {$\un{z}$} [ ] at 170 59
 \pinlabel {$\un{x}$} [ ] at 39 54
 \pinlabel {$-\frac{[k]}{[k-1]}$} [ ] at 120 56
\endlabellist
\centering
\ig{1}{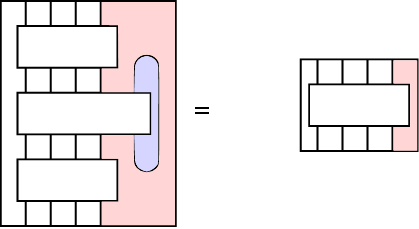}
}. \end{equation} To show this we use induction,
assuming that $JW(\un{x})$ was defined using either \eqref{eq:JWeasy} or \eqref{eq:JWhard} to extend $\un{z}$ by $b$. When $k=2$, $JW(\un{x})$ is defined
using \eqref{eq:JWeasy}, and \eqref{eq:reverselower} is clear since the value of a circle is $-[2]$. When $k>2$, $JW(\un{x})$ is defined using \eqref{eq:JWhard}, and the size of the tail of $\un{z}$ is $k-1$. Writing $JW(\un{x})$ as a linear combination of crossingless matchings, the only ones with nonzero contribution to \eqref{eq:reverselower} are the identity and the diagram in \eqref{rightcoeff}. The identity contributes $-[2]$ times $JW(\un{z})$, and the diagram in \eqref{rightcoeff} contributes $\frac{[k-2]}{[k-1]}$ times $JW(\un{z})$. Adding these, one obtains $\frac{[k-2] - [2][k-1]}{[k-1]} = \frac{-[k]}{[k-1]}$ times $JW(\un{z})$, as desired.

Suppose that $[k]$ is invertible. The RHS of \eqref{eq:JWhard} is obviously killed by any cup other than a cup on the final subsequence $rbr$. This final cup also kills the RHS, by
\eqref{eq:reverselower}. The coefficient of the identity in the RHS is only affected by the first term, and is therefore equal to 1. Thus the RHS of \eqref{eq:JWhard} is by definition
equal to $JW(\un{x}r)$. The statement about the orthogonal idempotent is also clear.

Now suppose that $[k]$ is not invertible. Multiplying the RHS of \eqref{eq:JWhard} by $[k]$, one obtains a map in $B$ which is well-defined. The coefficient of \eqref{rightcoeff} is
now $[k-1]$, which is invertible, but the coefficient of the identity is $[k]$, which is not invertible. If $JW(\un{x}r)$ exists then any element of $B$ is $JW(\un{x}r)$ multiplied by
the coefficient of the identity; if the coefficient of the identity is non-invertible, then every coefficient is non-invertible. This is a contradiction, so that $JW(\un{x}r)$ can not
exist. \end{proof}

\begin{remark} One could also prove this recursive formula using the usual recursive formula from \cite{Wenzl} for Jones-Wenzl projectors, combined with Proposition
\ref{prop:JWdescriptive} below. However, we felt this proof was still useful and motivational. \end{remark}

\begin{cor} \label{cor:grothSTL} Suppose that all quantum numbers are invertible. Let $V_{\un{x}}$ denote the image of $JW(\un{x})$, an indecomposable object of the Karoubi envelope
$\Kar(\STL)$. If $\un{x}$ ends in $rb$, and $\un{z}$ is the initial sequence missing only the final $b$, then one has \begin{equation} V_{\un{x}} V_{bg} \cong V_{\un{x}g}, \end{equation}
\begin{equation} V_{\un{x}} V_{br} \cong V_{\un{x}r} \oplus V_{\un{z}}. \end{equation} \end{cor}

\begin{proof} This is implied by \eqref{eq:JWeasy} and \eqref{eq:JWhard}, which give a decomposition of the identity of $V_{\un{x}} V_{bg}$ and $V_{\un{x}} V_{br}$ respectively into
orthogonal idempotents which factor through the appropriate objects. \end{proof}

\subsection{A descriptive formula for top idempotents} \label{subsec:descriptive}

The recursive formula of Proposition \ref{prop:JWrecursive} does not completely answer the question of when the map $JW(\un{x})$ exists. After all, it is possible for $JW(\un{x})$ to exist
even when $JW(\un{y})$ does not exist for an initial subsequence $\un{y} \subset \un{x}$. As an example, consider the case when $\un{x}$ is an alternating sequence, so that the question
reduces to the usual Temperley-Lieb algebra and its Jones-Wenzl projectors.

\begin{claim} \label{claim:binomials} Let $\un{x} = rbrb\ldots$ be an alternating sequence of length $k+1$. Then $JW(\un{x})$ exists if and only if the quantum binomial coefficients ${k
\brack m}$ are invertible in $\Bbbk$, for all $0 \le m \le k$. Equivalently, the Jones-Wenzl projector in the usual Temperley-Lieb algebra (on $k$ strands) exists if and only if the quantum
binomial coefficients are invertible. \end{claim}

This claim is fundamental to the theory of Temperley-Lieb algebras, but we have not been able to find it in the literature. A proof by Ben Webster can be found in the appendix.

\begin{remark} It was shown by Westbury \cite[Lemma 5]{WestburyTL} that, when all the quantum numbers $[m]$ are invertible for $0 \le m \le k$, then the Temperley-Lieb algebra is
semisimple. In other words, when these quantum numbers are invertible, then the Temperley-Lieb algebra has many idempotents, all the idempotents one can generically expect. The Jones-Wenzl
projector is just one of these idempotents, and its existence is a weaker condition. \end{remark}

\begin{example} $JW(rbrb)$ exists when $[3]$ is invertible. This can happen even when $[2]$ is not invertible (e.g., when $[2]=0$ one has $[3]=-1$), in which case $JW(rbr)$ does not exist.
\end{example}


Now we use Claim \ref{claim:binomials} to give an exact condition for whether $JW(\un{x})$ exists.

\begin{prop} \label{prop:JWdescriptive} Suppose that ${k \brack m}$ is invertible whenever $0 \le m \le k$ and $k+1$ is the length of a maximal alternating subsequence of $\un{x}$. Then
one has \begin{equation} \label{eq:JWsidebyside} 	{
	\labellist
	\small\hair 2pt
	 \pinlabel {$\un{x}$} [ ] at 75 20
	\endlabellist
	\centering
	\ig{1}{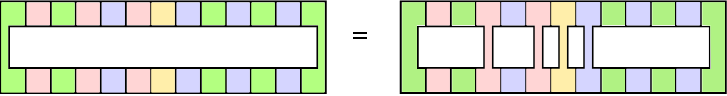}.
	} \end{equation} In this equation, the smaller rectangles which appear are the $JW$ maps associated to maximal alternating subsequences. On
the other hand, if some such ${k \brack m}$ is not invertible, then $JW(\un{x})$ does not exist. \end{prop}

This result should hardly be unexpected in light of Proposition \ref{prop:JWrecursive}, as this is exactly the morphism which the recursive formula of the previous section would construct.

\begin{proof} Under the assumptions of invertibility, clearly the RHS of \eqref{eq:JWsidebyside} exists, and clearly it satisfies the defining conditions of $JW(\un{x})$.

Now suppose that some ${k \brack m}$ is not invertible, where $k+1$ is the length of a maximal alternating subsequence $\un{y}$. By the algebraic proof of Claim \ref{claim:binomials},
there is some element $f$ of $B(\un{y})$ with a noninvertible coefficient of the identity, but with an invertible coefficient of some non-identity diagram $D$. Taking the horizontal
concatenation of $f$ with $JW(\un{z})$ for the other maximal alternating subsequences $\un{z}$ of $\un{x}$, one obtains an element of $B(\un{x})$ with a noninvertible coefficient of the
identity, but with an invertible coefficient of $\1 \ot D \ot \1$. This contradicts the existence of $JW(\un{x})$, using the same argument as in the end of the proof of Proposition
\ref{prop:JWrecursive}. \end{proof}

\subsection{Generalizations} \label{subsec:generalizations}

We no longer assume that $\Bbbk$ is a $\ZM[\d]$-algebra. Let $A = (a_{s,t})$ be an $S \times S$ matrix with values in $\Bbbk$. We will only be interested in the values of $a_{s,t}$ for
$s \ne t$. By convention, we set $a_{s,s}=2$, so that $A$ is a Cartan matrix (in the sense of the next chapter).

\begin{defn} Let $\STL(A)$ be the 2-category defined as in Definition \ref{defn:STL} except that instead of \eqref{eq:circle} one has \begin{equation} \label{eq:gencircle} 	{
	\labellist
	\small\hair 2pt
	 \pinlabel {$a_{b,r}$} [ ] at 60 20
	\endlabellist
	\centering
	\ig{1}{circle}.
	} \end{equation} \end{defn}

In other words, a circle still evaluates to a scalar, but which scalar depends on the color both inside and outside. When $\Bbbk$ is a $\ZM[\d]$-algebra, the special matrix with
$a_{s,t} = -\d$ for all $s \ne t$ will recover the original 2-category $\STL$.

For any two fixed colors $s \ne t \in S$, there is a notion of \emph{two-colored quantum numbers} $[m]_{s,t}$. These are defined by recursive formulae,
starting with $[0]_{s,t}=0$ and $[1]_{s,t}=1$. Then one has \begin{equation} \label{eq:2qbase} [2]_{s,t} = -a_{s,t}, \qquad \qquad [2]_{t,s} = -a_{t,s},
\end{equation} \begin{equation} \label{eq:2qrecursive} [2]_{s,t} [m]_{t,s} = [m-1]_{s,t} + [m+1]_{s,t}, \qquad \qquad [2]_{t,s} [m]_{s,t} = [m-1]_{t,s} +
[m+1]_{t,s}. \end{equation} It is not difficult to see that $[m]_{s,t} = [m]_{t,s}$ when $m$ is odd, and that when $m$ is even, $[m]_{s,t}$ is a multiple of
$[2]_{s,t}$ and $\frac{[m]_{s,t}}{[2]_{s,t}} = \frac{[m]_{t,s}}{[2]_{t,s}}$. There are many analogies between these two-colored quantum numbers and usual
quantum numbers; see \cite[Appendix]{EDihedral} for more details.

The two-colored quantum binomial coefficients ${k \brack m}_{s,t}$ are defined by the formula \begin{equation} {k \brack m}_{s,t} = \frac{[1]_{s,t} [2]_{s,t}
\cdots [k]_{s,t}}{[1]_{s,t} \cdots [m]_{s,t} [1]_{s,t} \cdots [k-m]_{s,t}}. \end{equation} By comparing the numbers of even and odd quantum numbers, it is
easy to see that ${k \brack m}_{s,t} = {k \brack m}_{t,s}$ unless $k$ is even and $m$ is odd, in which case ${k \brack m}_{s,t}$ is a multiple of $[2]_{s,t}$
and $\frac{{k \brack m}_{s,t}}{[2]_{s,t}} = \frac{{k \brack m}_{t,s}}{[2]_{t,s}}$.

The results of Claim \ref{claim:oneD} still hold. Proposition \ref{prop:JWrecursive} and its corollaries will still hold, after replacing quantum numbers with the two-colored
quantum numbers. More precisely, one should replace \eqref{eq:JWhard} with
\begin{equation} 	{
	\labellist
	\small\hair 2pt
	 \pinlabel {$\un{x}r$} [ ] at 43 53
	 \pinlabel {$\un{x}$} [ ] at 147 53
	 \pinlabel {$\un{x}$} [ ] at 273 86
	 \pinlabel {$\un{x}$} [ ] at 273 21
	 \pinlabel {$\un{z}$} [ ] at 262 53
	 \pinlabel {$\frac{[k-1]_{b,r}}{[k]_{r,b}}$} [ ] at 221 57
	\endlabellist
	\centering
	\ig{1}{JWhard}
	}. \label{eq:JWhard2q} \end{equation}
Corollary \ref{cor:grothSTL} holds once one replaces the first sentence with ``Suppose all two-colored quantum numbers are invertible, for all pairs $s \ne t \in S$." The analogue of Proposition \ref{prop:JWdescriptive} is:

\begin{cor} \label{cor:whenJWexists2q} Suppose that ${k \brack m}_{s,t}$ is invertible whenever $0 \le m \le k$ and there exists a maximal alternating
subsequence $\ldots ts \subset \un{x}$ of length $k+1$. Then $JW(\un{x})$ exists and \eqref{eq:JWsidebyside} holds. If one of these two-colored quantum
binomial coefficients is not invertible, then $JW(\un{x})$ does not exist. \end{cor}

\section{Universal Soergel bimodules}
\label{sec:universal}

\subsection{Universal Coxeter groups and Hecke algebras} \label{subsec:coxandhecke}

The \emph{universal Coxeter group} associated to a finite set $S$ is a group $W$ with a Coxeter presentation having generators $S$ and relations $s^2 = 1$ for each $s \in S$. An
\emph{expression} is a sequence $\un{x} = s_1 s_2 \ldots s_d$ of elements of $S$; it has length $d$ and is \emph{reduced} if $s_i \ne s_{i+1}$ for any $i$. Removing the underline, we let $x$ denote the corresponding product of generators in $W$. Each element of $W$ has a unique reduced expression.

The \emph{Hecke algebra} $\HB$ of $W$ is the $\Zvv$-algebra having a presentation with generators $H_s$ for $s \in S$, and relations \begin{equation} (H_s + v) (H_s - v^{-1}) = 0
\end{equation} for each $s \in S$. It has a \emph{Kazhdan-Lusztig} or \emph{KL basis} $\{b_w\}_{w \in W}$ as a $\Zvv$-module. One has $b_1 = 1$ and $b_s = H_s + v$ for each $s \in S$. Dyer \cite[Lemma 6.1]{DyerUniversal} has shown that the KL basis is given by the following recursive formula.

\begin{prop}\label{Dyer}(The Dyer formula) If $\un{x} = s_1 s_2 \ldots r b$ is a reduced expression, $g \in S$, $g \ne r$ and $g \ne b$, then \begin{equation} \label{eq:Dyereasy} b_x b_g = b_{xg}. \end{equation}
On the other hand, \begin{equation} \label{eq:Dyerhard} b_x b_r = b_{xr} + b_{xb}. \end{equation} \end{prop}

Note that the reduced expression for $xb$ is the initial subsequence $\un{z}$ of $\un{x}$ which is missing only the final $b$.

Whenever $\un{x}$ is a reduced expression but $\un{x}s$ is not, one can show that \begin{equation} b_x b_s = (v+v^{-1}) b_x. \end{equation} Between this equation and the Dyer formula,
one can compute the product of any $b_w$ for $w \in W$ with any $b_s$ for $s \in S$.

\subsection{Realizations} \label{subsec:realizations}

\begin{defn} A \emph{realization} of a universal Coxeter group $W$ over $\Bbbk$ is a free, finite rank $\Bbbk$-module $\hg$ together with its dual $\hg^*$, a choice
of \emph{simple roots} $\{\a_s\}_{s \in S} \subset \hg^*$ and a choice of \emph{simple coroots} $\{\a_s^\vee\}_{s \in S} \subset \hg$, satisfying $\langle \a_s, \a_s^\vee \rangle =2
\in \Bbbk$. \end{defn}

The \emph{Cartan matrix} attached to a realization is the $S \times S$ matrix $A = (a_{s,t})$ with values in $\Bbbk$, given by $a_{s,t} = \langle \a_t, \a_s^\vee \rangle$. We do not
assume that the simple roots span $\hg^*$ or that the simple coroots span $\hg$, so that $A$ need not determine the realization.

We assume in this paper that our realization satisfies \emph{Demazure surjectivity}, which is the condition that $\langle \a_s, \cdot \rangle$ and $\langle \cdot, \a_s^\vee \rangle$
are both surjective maps to $\Bbbk$.

Given any realization, there is an action of $W$ on $\hg$ defined on generators by the formula $s(v) = v - \langle \a_s, v \rangle \a_s^\vee$. The contragredient action on $\hg^*$ is
given by $s(f) = f - \langle f , \a_s^\vee \rangle \a_s$. Let $R$ be the coordinate ring of $\hg$, or in other words, the $\Bbbk$-linear polynomial ring whose linear terms are $\hg^*$.
We give $R$ a grading, so that $\deg(\hg^*)=2$. The commutative ring $R$ has a natural homogeneous action of $W$.

There is a Demazure map $\pa_s \co R \to R^s$, defined by the formula
$$\pa_s=\frac{f-s(f)}{\alpha_s},$$
whose image is the set of $s$-invariant polynomials. On linear polynomials $f \in \hg^* \subset R$, one has $\pa_s(f) = \langle f,
\a_s^\vee \rangle \in \Bbbk$.

\subsection{Diagrammatics for Soergel bimodules} \label{subsec:diagrammatics}

Instead of defining and developing the theory of Soergel bimodules, we prefer to follow the diagrammatic approach developed in \cite{EWGR4SB}. Our object of study will be a certain
monoidal category with graded Hom spaces, given diagrammatically. 

\begin{defn} A \emph{Soergel graph} is a certain kind of finite graph embedded in the planar strip $\RM \times [0,1]$. The edges in this graph are colored by $s \in S$. The only vertices allowed
are trivalent vertices connecting three edges of the same color, and univalent vertices (called \emph{dots}). We also allow edges which meet no vertices, forming a circle. This graph is
allowed to have a boundary on the walls of the strip (i.e. edges may terminate at $\RM \times \{0\}$ or $\RM \times \{1\}$, though these termination points are not counted as vertices).
The edge labels that meet the boundary give two sequences of colors, the \emph{bottom boundary} and the \emph{top boundary}. A  \emph{region} of the graph is a connected component of the complement of the Soergel graph in $\RM \times [0,1]$. Finally, we may place a homogeneous polynomial in $R$ inside
each region of the graph. We consider these graphs up to isotopy, though this isotopy must preserve $\RM \times \{0\}$ and $\RM \times \{1\}$. A Soergel graph has a \emph{degree}, which
accumulates $+1$ for every dot, $-1$ for every trivalent vertex, and the degree of each polynomial. \end{defn}

In particular, the connected components of a Soergel graph have a single color. For numerous examples, look ahead.

\begin{defn} \label{defn:DC} Let $\DC$ be the monoidal category defined as follows. The objects are monoidally generated by $s \in S$, so that a general object is an expression
$\un{x}$ (not necessarily reduced, possibly empty). Given two expressions $\un{x}$ and $\un{y}$, the morphism space $\Hom(\un{x},\un{y})$ will be the $\Bbbk$-module spanned by Soergel graphs with bottom boundary $\un{x}$ and top boundary $\un{y}$, modulo the local relations below. This morphism space is graded by the degree of the Soergel graphs,
and all the relations below are homogeneous.

The \textbf{Needle relation}:
\begin{equation} \label{eq:needle} \ig{1}{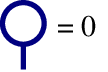} \end{equation}

The \textbf{Frobenius relations}:
\begin{subequations} \label{thefrobeniusrelations}
\begin{equation} \label{eq:assoc} \ig{1}{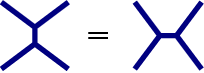} \end{equation}	
\begin{equation} \label{eq:unit} \ig{1}{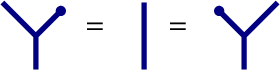} \end{equation}	
\end{subequations}

The \textbf{Barbell relation}:
\begin{equation} \label{eq:barbell} {
\labellist
\small\hair 2pt
 \pinlabel {$\a_b$} [ ] at 32 12
\endlabellist
\centering
\ig{1}{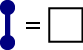}
} \end{equation}

The \textbf{Polynomial forcing relation}:
\begin{equation} \label{eq:dotforcegeneral} {
\labellist
\small\hair 2pt
 \pinlabel {$f$} [ ] at 8 12
 \pinlabel {$\pa_b(f)$} [ ] at 66 13
 \pinlabel {$b(f)$} [ ] at 120 13
\endlabellist
\centering
\ig{1}{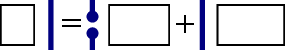}
} \end{equation}

This ends the definition. \end{defn}

We may also consider Soergel graphs on the planar disk; these have a single boundary sequence $\un{x}$, which is to be considered only up to cyclic permutation. A Soergel graph on the disk does not
represent a morphism in $\DC$. However, as the relations above are local, one can apply them to any disk within the planar strip, so disk diagrams are useful for local calculations.

Let $\Kar(\DC)$ denote the Karoubi envelope of the graded, additive closure of $\DC$. The following theorem is proven in \cite{EWGR4SB}, whose analogue for Soergel bimodules is known
as the \emph{Soergel Categorification Theorem}.

\begin{thm} \label{thm:SCT} The indecomposable objects in $\Kar(\DC)$, up to isomorphism and grading shift, can be labeled by $w \in W$. The indecomposable object $B_w$ is the unique
summand inside $\un{w}$ for a reduced expression of $w$ which is not a direct summand of $\un{y}$ for any shorter expression. There is an isomorphism of $\Zvv$-algebras \[ \HB \to
[\Kar(\DC)] \] from the Hecke algebra to the Grothendieck ring of $\Kar(\DC)$, which sends $b_s$ to the symbol of the generating object $s$, which is equal to $B_s$. \end{thm}

Soergel conjectured that, when $\Bbbk$ has characteristic zero and the representation $\hg$ is ``reflection-faithful," this isomorphism sends $b_w$ to $[B_w]$. Our goal in the rest of this
paper is to give a criterion for when $b_w \mapsto [B_w]$, which will happen when we can categorify the Dyer formula.

\subsection{Maximally connected graphs and minimal degrees} \label{subsec:maxcon}

This section is an adaptation of \cite[section 5.3.3]{EDihedral}, which only treated the case of two colors. However, the arguments are almost identical.

Using the relations in Definition \ref{defn:DC}, it is not hard to show that every graph is in the span of a graph containing only \emph{simple trees} with
polynomials. In other words, each connected component of the graph is a tree with non-empty boundary. Any two trees with the same boundary are equal by
\eqref{eq:assoc} and \eqref{eq:unit}. Moreover, this tree contains a dot precisely when the boundary is a single point, in which case the tree has no
trivalent vertices; we call such a tree a \emph{boundary dot}. Moreover, one can assume that there is a single polynomial, and it occurs in a region of one's
choosing (say, the leftmost region). The proof in \cite[Proposition 5.19]{EDihedral} works verbatim.

\begin{defn} A Soergel graph containing only simple trees is \emph{maximally connected} if it has no polynomials, and satisfies the following condition for each $s \in S$. Consider the
subgraph $\Gamma$ consisting of all the edges colored by $S \setminus s$. Then $\Gamma$ splits the planar strip into regions, and each region may contain at most one connected
component colored $s$. \end{defn}

It is easy to see that maximally connected Soergel graphs with a given boundary exist, and that for any graph which is not maximally connected, one can produce a maximally connected
graph of smaller degree by ``fusing" two edges (see \cite[section 5.3.3]{EDihedral}) or removing a polynomial. However, not all maximally connected graphs have the same degree, as the following
examples show.

\igc{1}{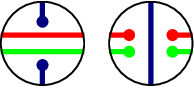}

Given a $S$-colored crossingless matching on the planar disk (with at least two colors), one can obtain a maximally connected Soergel graph by taking a deformation retract of each colored
region. A quick inductive argument shows that any such Soergel graph has degree $+2$. The choice of deformation retract is irrelevant, because any two trees with the same boundary are
equal.

\igc{.75}{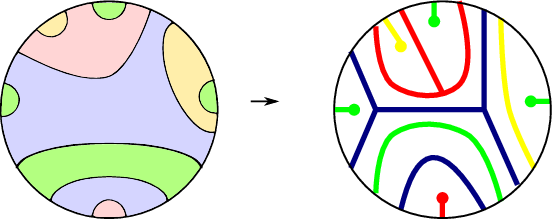}

The regions in the resulting graph (RHS) correspond to the strands in the original crossingless matching (LHS). Therefore, each region is bounded by exactly two colors, and meets the
boundary exactly twice. It is easy to recover the colored crossingless matching from the resulting graph: simply deformation retract each region into a strand, and use the colors of
the graph to color the regions between strands.

\begin{prop} \label{prop:degreeofmaxcon} Let $\un{x} = s_1 s_2 \ldots s_d$ be a sequence representing the boundary of a Soergel graph on the planar disk (so that we only consider
$\un{x}$ up to rotation). Then any maximally connected Soergel graph on the disk with boundary $\un{x}$ has degree $\ge 2-m$, where $m$ is the number of repetitions in $\un{x}$ (i.e.
the number of $1 \le i \le d$ such that $s_i = s_{i+1}$, where we set $s_{d+1}=s_1$). If $\un{x}$ has no repetitions and the graph has degree $2$, then it arose as the deformation
retract of some colored crossingless matching. \end{prop}

\begin{proof} It is easy to reduce to the case where $\un{x}$ has no repetitions.

The maximally connected graph splits the disk into regions. Since there are no cycles, each region must meet the boundary of the disk at least once, say between $s_i$ and $s_{i+1}$.
Suppose that a region $X$ meets the boundary of the disk $k$ times, so that there are distinct indices $i_1, i_2 ,\ldots, i_k$ such that $X$ meets the boundary between $s_{i_j}$ and
$s_{i_j + 1}$. By following the walls of $X$ we see that the colors $s_{i_j + 1}$ and $s_{i_{j+1}}$ are equal, and the maximally connected condition implies that the $k$ colors
$s_{i_j}$ are all distinct. Therefore, the number of bordering colors of a region is equal to the number of times that region meets the boundary. This number is always at least 2,
since there are no repetitions in $\un{x}$.

If each region meets the boundary exactly twice, then we can deformation retract the regions into strands to obtain a colored crossingless matching, as above. It is easy to see that
this yields a bijection between $S$-colored crossingless matchings, and maximally connected graphs where each region meets the boundary exactly twice.

If a region meets the boundary $k$ times, one can use this region to cut the overall graph into $k$ subgraphs, each of which is either a boundary dot, or a maximally connected graph
with one repetition. This is illustrated in the picture below, for a region $X$ meeting the boundary 4 times. Using induction therefore, each subgraph has degree at least $+1$. If any
region meets the boundary $\ge 3$ times, it follows that the overall degree of the original graph is $\ge 3$. \end{proof}

\igc{1}{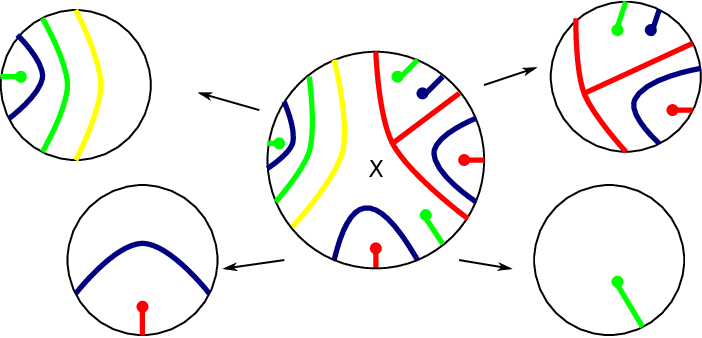}

Proposition \ref{prop:degreeofmaxcon} allows one to place a lower bound on the degree of the Hom space between two objects in $\DC$.

\begin{cor} \label{cor:lowerbound} Let $\un{x} = s_1 \ldots s_d$ and $\un{y} = t_1 \ldots t_k$ be two nonempty reduced expressions. If $s_1 = t_1$ and $s_d = t_k$ then every nonzero
morphism in $\Hom_{\DC}(\un{x},\un{y})$ has degree $\ge 0$. Otherwise, every nonzero morphism has degree $\ge 1$. Similarly, nonzero morphisms in $\Hom_{\DC}(\un{x},\emptyset)$ and
$\Hom_{\DC}(\emptyset,\un{y})$ have degree $\ge 1$, while every nonzero morphism in $\Hom_{\DC}(\emptyset,\emptyset)$ has degree $\ge 0$. \end{cor}

\begin{proof} Let $\un{y}^{\op}$ denote the sequence $\un{y}$ in reverse. Viewing $\un{x} (\un{y}^{\op})$ as a long sequence on the circle, there is one repetition if $s_1=t_1$, and
one repetition if $s_d = t_k$. By Proposition \ref{prop:degreeofmaxcon}, the minimal degree of any map is at least $2$ minus the number of repetitions. Similarly, if $\un{y}$ is empty,
then $\un{x}$ has at most one repetition, if $s_1 = s_d$. \end{proof}

Given an $S$-colored crossingless matching on the planar strip, one can obtain a maximally connected Soergel graph of degree $0$ by deformation retract, as in the example below.

\igc{1}{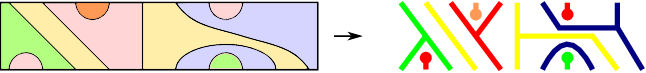}

\begin{cor} \label{cor:TLknows0} Every degree zero map between nonempty reduced expressions in $\DC$ arises from an $S$-colored crossingless matching on the planar strip. \end{cor}

The lower bound in Corollary \ref{cor:lowerbound} can also be obtained from Soergel's Hom formula. The advantage of this approach is an explicit description of the morphisms in the lowest
degree. This description can also be obtained, with some work, from the second author's light leaves basis for Hom spaces.

\subsection{The main theorem} \label{subsec:mainthm}

\begin{prop} \label{prop:mainprop} Let $A$ be the Cartan matrix of the realization. There is a non-monoidal functor from $\STL(A)$ (viewed as a 1-category) to $\DC$, which sends the
1-morphism $\un{x}$ in $\STL(A)$ to the object $\un{x}$ in $\DC$, and sends a 2-morphism in $\STL(A)$ corresponding to an $S$-colored crossingless matching to the corresponding degree $0$
deformation retract. This functor is essentially surjective (in the Karoubi envelope) and fully faithful onto maps of degree $0$. \end{prop}

\begin{proof} First we must show that this deformation retract map preserves the algebra structure. It is easy to check \eqref{eq:isotopy}. Relation \eqref{eq:gencircle} follows from
$\pa_b(\a_r) = a_{b,r}$, a deduction we leave as an exercise to the reader. Thus the functor is well-defined. Corollary
\ref{cor:TLknows0} implies that it is full onto degree $0$. Soergel's Hom formula implies that the dimensions of Hom spaces agree, and thus it is also faithful. By the classification of
Theorem \ref{thm:SCT}, each indecomposable object in $\Kar(\DC)$ appears as a summand of an object in the image of the functor, so the functor is essentially surjective. \end{proof}

This non-monoidal functor can be upgraded to a genuine 2-functor to the 2-category of singular Soergel bimodules. For definitions and a proof in the two-color case, see \cite[section 5.2.3]{EDihedral}.

By Proposition \ref{prop:mainprop}, every idempotent in $\End_{\DC}(\un{x})$ for a nontrivial reduced expression comes from an idempotent in $\End_{\STL(A)}(\un{x})$. In particular, the
theory of top idempotents implies that, when $JW(\un{x})$ exists, it must project to the indecomposable object $B_x \in \Kar(\DC)$. Our main theorem is now immediately implied by Corollary
\ref{cor:grothSTL}.

\begin{thm} Suppose that all two-colored quantum numbers are invertible in $\Bbbk$. Then for a reduced expression $\un{x} = \ldots r b$ and a simple reflection $s \in S$ one has the
following isomorphism in $\Kar(\DC)$, which categorifies the Dyer formula. As a consequence, the map $\HB \to [\Kar(\DC)]$ sends $b_w \mapsto [B_w]$.
\begin{equation}\label{Soso} B_x B_s \cong \begin{cases} B_x(1) \oplus B_x(-1) & \textrm{if} \; s=b, \\ B_{xs} & \textrm{if} \; s \ne r,b, \\ B_{xr} \oplus B_{xb} & \textrm{if} \; s=r. \end{cases}
\end{equation} \end{thm}

\begin{proof} The case when $s = b$ follows from the case when $x=s$, and was proven in \cite{EWGR4SB}. The other cases follow from the recursion formula for top idempotents.
\end{proof}

Moreover, when certain two-colored quantum binomial coefficients vanish, it is immediately clear for which $w$ the statement that $b_w \mapsto [B_w]$ will fail, as in Corollary
\ref{cor:whenJWexists2q}.

\begin{remark} The vanishing of two-colored quantum numbers determines which Coxeter quotient of $W$ acts faithfully on $\hg$, as was discussed in \cite[Appendix]{EDihedral}. In
particular, if the universal Coxeter group $W$ acts faithfully on $\hg$, then all two-colored quantum numbers are nonzero. Therefore, $b_w \mapsto [B_w]$ is always satisfied for a
faithful realization over a field $\Bbbk$. \end{remark}

\begin{remark} \label{rmk:typical} The typical crystallographic setting for universal Coxeter groups is the realization where $\Bbbk = \ZM$ and $a_{s,t} = -2$ for all $s \ne t$. In this
case, the two-colored quantum number $[m]_{s,t}$ is equal to the integer $m$. Specializing to a field of finite characteristic, it is clear which binomial coefficients vanish.
\end{remark}

\begin{remark} It is not difficult, using Soergel's Hom Formula (see \cite{EWGR4SB}), to extend these results to those reduced expressions $\un{x}$ in arbitrary Coxeter groups $W$ for
which every element $y \le x$ in $W$ has a unique reduced expression. For general Coxeter groups, morphisms are not spanned by univalent-trivalent Soergel graphs as above, requiring a
more complicated definition of Soergel graphs. However, the Hom formula implies that univalent-trivalent graphs are sufficient for these kinds of expressions. \end{remark}

\begin{remark} In unpublished work, the second author proved Soergel's conjecture for universal Coxeter groups when $\Bbbk=\RM$, while considering the wider study of ``large" Coxeter
groups. The proof involved singular Soergel bimodules, and produced indecomposable bimodules using a formula analogous to \eqref{eq:JWsidebyside}. For more details, contact the second
author. \end{remark}

The isomorphism \eqref{Soso} categorified what we have called the Dyer formula. However, Dyer \cite{DyerUniversal} produces several other formulas for Hecke algebras of universal Coxeter
groups. Most of these can also be deduced fairly easily from our main theorem. Let us sketch the connections here; to go into any more depth would require too much notation, but the avid
reader should be able to draw the correct conclusion. We assume below that all two-colored quantum numbers are invertible.

Dyer's formula \cite[(3.12)]{DyerUniversal} deals with the decomposition of $B_v B_w$ into indecomposables. There are two cases. In the first case, $\ell(v) + \ell(w) = \ell(vw)$.
Proposition \ref{prop:JWdescriptive} implies that placing $JW_v$ next to $JW_w$ will result in $JW_{vw}$, except when the tail of $v$ overlaps with the ``head" (i.e. maximal alternating
initial subsequence) of $w$ to produce a longer alternating subsequence of $vw$. What happens to the overlap in this case is exactly what happens in the representation theory of $\sl_2$
when one takes the tensor product of two irreducible modules; it is governed by the Clebsch-Gordan formula. One can observe that Dyer's sums $C(w,i)$ are essentially a reformulation of the
Clebsch-Gordan formula. In the second case, $\ell(v) + \ell(w) > \ell(vw)$, because some simple reflection $s$ appears on the right of $v$ and the left of $w$. In this case, a factor of
$v+v^{-1}$ will appear in the middle, as in the decomposition $B_s B_s \cong B_s(1) \oplus B_s(-1)$. This feature of $\DC$ has nothing to do with the multicolored Temperley-Lieb category,
dealing instead with the morphisms of non-zero degree in the Soergel category, but it is amply discussed in other papers (e.g. \cite{EWGR4SB}). Taking this into account, it is easy to
reduce to the first case.

Dyer's formula \cite[(4.1)]{DyerUniversal} computes the relative KL polynomial $P_{xw}^y$. Let us discuss the case when $y=1$, which gives the ordinary KL polynomial $P_{x,w}$. We assume
that the reader is familiar with the light leaves basis for morphisms in $\DC$, and the corresponding notation for subexpressions (see \cite{EWGR4SB}). Let $\un{w}$ be the unique reduced
expression for $w$. We say that a subexpression $\eb$ of $\un{w}$ contains a \emph{pitchfork} if it contains the configuration $sts$ with $s \ne t \in S$, where $t$ is U0 and the second $s$
is either D1 or D0. The corresponding light leaves map contains a pitchfork, as in \cite[(5.17)]{EWGR4SB}. This is the morphism in $\DC$ corresponding to a cap in the multicolored
Temperley-Lieb category, and thus will be killed by precomposition with the Jones-Wenzl idempotent. We call a subexpression \emph{pitchfork-avoiding} if it does not contain a pitchfork, and
thus its light leaves map survives precomposition with the Jones-Wenzl idempotent; a standard localization argument shows that the pitchfork-avoiding light leaves maps remain linearly
independent after this precomposition. It remains to observe that Dyer's set $\PC_w(1,x)$ is precisely the set of subexpressions $\eb$ of $\un{w}$ which express $x$, and which are
pitchfork-avoiding. Moreover, Dyer's number $\rho_w(\eb)$ is a computation of the defect of the light leaves map (up to an overall renormalization). The defects of these light leaves maps
compute the diagrammatic character at $x$ of the indecomposable $B_w$ (see \cite[Definition 6.23]{EWGR4SB}), which agrees with the KL polynomial $P_{x,w}$ (up to renormalization).


\appendix
\newcommand{\qbinom}[2]{{\genfrac{[}{]}{0pt}{}{#1}{#2}}}
\newcommand{\TL}{\mathcal{T\!L}}

\section{The existence of Jones-Wenzl projectors (by Ben Webster)}

Consider the Temperley-Lieb category $\TL$ over $\Z[\delta]$ for the parameter $\delta$, and the Temperley-Lieb algebra $TL_n=\End_{\TL}(n,n)$. Let $I_{<n}$ be the ideal in $TL_n$ spanned
by all morphisms which factor through the object $k$ for $k<n$. Given a commutative ring $R$, and a homomorphism $\Z[\delta]\to R$ (that is, a choice of $\delta\in R$), we call an
idempotent $J_n\in TL_n\otimes_{\Z[\delta]}R$ a {\em Jones-Wenzl projector} if $J_n I_{<n} = I_{<n} J_n = 0$ and $1 - J_n \in I_{<n}$.

\begin{prop} For any $\delta\in R$, there is at most one Jones-Wenzl projector. \end{prop}

\begin{proof} Assume that $J_n$ and $J_n'$ are two JW projectors. Then $J_n J_n' = J_n$, since $J_n' \equiv 1$ modulo $I_{<n}$. Similarly, $J_n J_n' = J_n'$. Thus $J_n =
J_n'$. \end{proof}

Thus we can speak of ``the'' Jones-Wenzl projector. The question we will wish to consider is when this projector exists and when it does not. For example, in $TL_2$, one can easily check
that the JW projector exists if and only if $\delta$ is a unit. For higher ranks, the existence question is more subtle, but still easy to resolve. As noted earlier, we can define unique
polynomials in $\delta$ that are sent to the quantum integers $[k]$ or quantum binomial coefficients $\qbinom{n}{k}$ under the specialization $\delta\mapsto q+q^{-1}$.

\begin{thm}\label{thm:JW-exist} The Jones-Wenzl projector exists over $R$ if and only if the quantum binomial coefficient $\qbinom nk$ is invertible in $R$ for all $k<n$. \end{thm}

 Let us give an outline of the proof, which uses the representation theory of a $\Z[q,q^{-1}]$-linear form $U$ of the quantum group $U_q(\sl_2)$, and of algebras $U'$ obtained
from this form by base change. One proves that (after base change) $\TL$ describes the morphisms between tensor products of the standard representation $V$ of $U'$. { For instance, this describes the image of $I_{<n} \subset \End(V^{\ot n})$ as the sum of the images of all maps from $V^{\ot k}$ to $V^{\ot n}$, for $k < n$.}
A Jones-Wenzl projector exists if and only if there is a
decomposition of $V^{\ot n}\cong \ker(I_{<n})\oplus
\operatorname{im}(I_{<n})$, in which case $J_n$ is the idempotent
projecting to $ \ker(I_{<n})$.  

The kernel $\ker(I_{<n})$ can be written as the image of a map $W(n) \to V^{\ot n} $ from the Weyl module of highest weight $n$, and $\operatorname{im}(I_{<n})$ as the kernel of a map to
the dual Weyl module $W^*(n)$. Consequently, there is a Jones-Wenzl projector if and only if the composition $W(n) \to V^{\ot n} \to W^*(n)$ is an isomorphism. By explicit computation, this
map is an isomorphism if and only if $\qbinom nk$ is invertible in $R$ for all $k < n$.

Unfortunately, while the literature does contain many of the results above when $q$ is generic or even a root of unity, it is harder to find statements which hold for arbitrary commutative
$\Z[q,q^{-1}]$ algebras. The paper \cite{DPS} of Du, Parshall, and
Scott studies Schur-Weyl duality in this generality, and contains
proofs of many of the desired statements. The  identification of $\operatorname{im}(I_{<n})$ as the
kernel of a map to the dual Weyl module $W^*(n)$ does not seem to appear, however, and its proof requires some development. Since the framework will be developed anyway, we give easier
proofs of some of the facts that could be quoted from \cite{DPS} as well. The techniques should be familiar to the experts. The final argument, that a Jones-Wenzl projector exists if and
only if the natural map $W(n) \to W^*(n)$ is an isomorphism, seems to be new.

We first justify passage to the base ring
  $\Z[q,q^{-1}]$. Let \[R'=R[q,q^{-1}]/(\delta-q-q^{-1}).\] Note that $R$ injects in $R'$ as the fixed points of the bar involution
sending $q\mapsto q^{-1}$ and fixing $R$. The Temperley-Lieb algebra $TL_n\otimes_{\Z[\d]}R'$ has an induced bar involution preserving the diagram basis whose fixed points are
$TL_n\otimes_{\Z[\delta]}R$. The uniqueness of $J_n$ guarantees that if it exists over $R'$ then it is preserved by the bar involution, and thus exists over $R$. Also, $\qbinom nk$ is
invertible in $R$ if and only if it is invertible in $R'$.

Let $U=U_q(\mathfrak{sl}_2)$ be the algebra generated over $A=\Z[q,q^{-1}]$ by $E^{(i)}$, $F^{(i)}$, and $K$ under the usual relations. We make $A$ a $\Z[\d]$-algebra by setting
$\delta=q+q^{-1}$. Note that a homomorphism $\Z[\delta]\to R$ induces
a unique map $A\to R'$ sending $q\to q$.  We let $U' = U \ot_A R'$, and for any $U$-module $M$ which is
free as an $A$-module, we write $M'$ for $M \ot_A R'$.

Consider the standard representation on $V=A^2$ via the matrices
\[E\mapsto
\begin{bmatrix}
  0 & 1\\
  0 & 0
\end{bmatrix}\qquad F\mapsto
\begin{bmatrix}
  0 & 0\\
  1 & 0
\end{bmatrix}
\qquad K\mapsto
\begin{bmatrix}
  q & 0\\
  0 & q^{-1}
\end{bmatrix}.
\]
The algebra $U$ is a Hopf algebra, with the coproduct
\begin{equation} \Delta(E^{(m)})=\sum_{r+p=m}q^{pr}E^{(r)}\otimes
  E^{(p)}K^{r}\qquad \Delta(F^{(m)})=\sum_{r+p=m} q^{-pr} K^{-p} F^{(r)}\otimes F^{(p)}.\label{eq:coproduct}
\end{equation}
Thus, there is an induced $U$-module structure on $M\otimes_AN$ for any $U$-modules $M,N$, and similarly for the base change to any $A$-algebra. In particular, we can also consider the
tensor power $V^{\ot n}$ (tensor product over $A$), and its base change $(V^{\ot n})'$.

Let $W(k)$ be the Weyl module over $A$, that is, the free $A$-module spanned by a highest weight vector $v_k$ of weight $k$, and its images under divided powers $F^{(i)}v_k$ for $0 \le i
\le k$, with $F^{(k+1)}v_k=0$. The algebra $U$ acts via
\[E^{(j)}F^{(i)}v_k=\qbinom{k-i+j}{j}F^{(i-j)}v_k \qquad F^{(j)}F^{(i)}v_k=\qbinom{j+i}{j}F^{(i+j)}v_k.\]  
There is a duality on $U$ modules which are free of finite rank over $A$ by considering the action on $M^*:=\Hom_A(M,A)$ induced by the antipode.
We let $W^*(k)$ be the dual Weyl module; this is spanned by the dual basis $\{w_{k-2p}\}$ for ${0\leq p \leq k}$ to the basis $\{F^{(p)}v_k\}$.

Note that $W(k)$ is also the quotient of $U$ by the left ideal generated by $E^{(m)}$ for $m > 0$, by $F^{(p)}$ for $p > k$, and by $K - q^k$. Thus, the following lemma is
obvious.

\begin{lemma}\label{lem:weyl-universal} The Weyl module $W(k)$ has the universal property that $\Hom_{U }(W(k),M)$ is canonically isomorphic to the set of vectors in the module $M$ of
weight $k$ killed by $E^{(m)}$ for $m>0$ and $F^{(p)}$ for $p>k$. Dually, the module $W^*(k)$ also has a universal property in the category of $U$-modules with underlying $A$-module free;
$\Hom_{U }(M,W(k)^*)$ is canonically isomorphic to space of $A$-homomorphisms $M\to A$ of weight $k$ killed by $E^{m}$ for $m>0$ and $F^{(p)}$ for $p>k$. \end{lemma}


The following lemma is standard (for example, see \cite[Proposition 5.4--5]{DPS}):

\begin{lemma} \label{lemma:stdcostdhom}
For all $k,l \ge 0$, we have 
\begin{align*} \Hom_{U' }(W(k)',W^*(l)')&= \begin{cases} R'&k=l\\ 0& k\neq l \end{cases} .\\ \Ext^1_{U' }(W(k)',W^*(l)')&=0. \end{align*}
\end{lemma}


\begin{defn} We say that a $U'$-module $M'$ is {\em Weyl filtered} if it possesses a filtration $L_0=0\subset L_1\subset L_2\subset \cdots \subset L_r=M'$ with $L_{j}/L_{j-1}\cong W(p_j)'$.
We say that $M'$ is {\em dual Weyl filtered} if it possesses such a filtration with $L_{j}/L_{j-1}\cong W^*(q_j)'$. We say it is {\em tilting} if is both Weyl filtered and dual Weyl
filtered. \end{defn} Note that these properties are preserved under base change, by the freeness of $W(k)$ over $A$.


\begin{lemma} The module $V^{\ot n}$ is tilting. So is $(V^{\ot n})'$. \end{lemma}

\begin{proof} First, note that $V^{\ot n}$ is self-dual, since $V$ is as well. Thus, it suffices to show that $V^{\ot n}$ has a Weyl filtration, or more generally that if $N$ has a
Weyl filtration, $N\otimes V$ does as well.

To see this, it suffices to show that for any Weyl module $W(k)$, we have a short exact sequence \[0\to W(k+1)\to W(k)\otimes V\to W(k-1)\to 0.\] Since $V=W(1)$, we denote its highest
weight vector by $v_1$. The inclusion $W(k+1)\hookrightarrow W(k)\otimes V$ sends
\[F^{(i)}v_{k+1}\mapsto F^{(i)}(v_k\otimes v_1)=F^{(i)}v_k\otimes v_1 +q^{-k+i-1}F^{(i-1)}v_k\otimes Fv_1.\]
The quotient module $Q$ has a basis as a free $A$-module given by the images of the vectors $F^{(p)}v_k\otimes Fv_1$ for $p=0,\dots, k-1$.
Since $Q$ is generated by the highest weight vector $v_k\otimes Fv_1$, there is a map $W(k-1)\to Q$, which is a surjective map between free $A$-modules of rank $k$ and thus an isomorphism.

The proof for $(V^{\ot n})'$ is identical.
\end{proof}

In particular, the multiplicities of (dual) Weyl modules in the (dual)
Weyl filtration on $V^{\ot n}$ agree with the multiplicities in the semisimple case (i.e. with base ring $R'=\QM(q)$). For another
proof of this lemma, see \cite[Proposition 5.4]{DPS}.

We can define a functor $\rho \colon \TL \to U\operatorname{-mod}$ which sends $n\mapsto V^{\ot n}$. This functor will be monoidal, so we need only specify the image of the cup
$\iota\colon A\to T_2$ and the cap $\epsilon\colon T_2 \to A$. These are given by the unique homomorphisms such that
\begin{equation}
\iota(1)=q^{-1}Fv_1\otimes v_1-v_1\otimes Fv_1 \qquad \epsilon(-Fv_1\otimes v_1)=\epsilon(q^{-1}v_1\otimes Fv_1)=1.\label{eq:cup-cap}
\end{equation}
The existence of this functor is a standard result due (in different form) to Temperley and Lieb's original paper \cite{TemLie}.  Obviously, we can extend scalars to obtain a functor
$\rho' \co \TL\otimes_{\Z[\d]} R'\to U'\operatorname{-mod}$.

\begin{lemma} The functor $\rho'$ is fully faithful. \end{lemma}

This lemma can be compared to \cite[Theorem 6.2]{DPS}. 

\begin{proof} Using duality, it suffices to prove that the induced map from $\Hom_{\TL}(0,n)\to \Hom_{U' }(R',(V^{\ot n})')$ is an isomorphism. By localization and Nakayama's lemma, it suffices
to prove this isomorphism after base change to any field $\Bbbk$.The dimension of $\Hom_{U' }(R', (V^{\ot n})')$ can be computed using Lemma \ref{lemma:stdcostdhom}, and it agrees
with the dimension of $\Hom_{\TL}(0,n)$; it is the $n$th Catalan number. Thus, it suffices to show that $\rho \ot \Bbbk$ is injective, that is, that the vectors $v_C$ attached to different
crossingless matchings $C$ by \eqref{eq:cup-cap} are linearly independent.

For a fixed $C$, we let $\epsilon_1,\cdots, \epsilon_n$ be the sequence of $n$ elements of $\{1,0\}$ where we put a $1$ over the left end of a cup and $0$ over the right end (so we have
$(1,1,0,0)$ for two nested cups, and $(1,0,1,0)$ for unnested). We have $v_C =q^{-n}F^{\epsilon_1}v_1\otimes \cdots \otimes F^{\epsilon_n}v_1+\cdots$ where the other terms correspond to
words in $\{1,0\}$ which are smaller in lexicographic order. This shows that no multiple $av_C$ be written in terms of $v_{C'}$ for $C'<C$ in lexicographic order. Thus, by
upper-triangularity, the vectors $v_C$ are linearly independent.
This completes the proof. \end{proof}

Given a $U$ module $M$, let $M_p$ denote its $p$-th weight space. Let $M[<n]$ denote the maximal submodule whose weight spaces for $p \geq n$ are zero.

\begin{prop} Assume that $M$ is a $U$ module with a dual Weyl filtration. Then the sum of the images of all maps $V^{\ot k}\to M$ for $k<n$ is precisely $M[<n]$. Furthermore, if $N$ is Weyl
filtered, then any map $N\to M[<n]$ is a sum of maps factoring through
$V^{\ot k}$ for $k<n$. The same result holds for $U'$ modules. \end{prop}

\begin{proof} We prove the result for $U$; the proof for $U'$ is identical. Assume $X$ is a submodule of $M$ with all weight spaces for $p\geq k$ trivial. We wish to show that
$X$ is in the sum of the images of maps $V^{\ot k} \to M$ for $k<n$. Let $Y$ be a submodule with a dual Weyl filtration $L_0=0\subset L_1\subset L_2\subset \cdots \subset L_r=Y$ satisfying $X
\subset Y \subset M$. As before, we set $L_i / L_{i-1} = W^*(q_i)$. We
choose $Y$ so as to minimize the length $r$, and prove the result by induction on $r$.

Let $p$ be the maximal weight of $Y$. Let $Q$ be any quotient of the $A$-module $Y_p$ which is free as an $A$-module. By the universal property, we have a map $Y\to Q
\boxtimes_{A} W^*(p)$. The intersections of $Y_p$ with $L_i$ form a filtration of $Y_p$ with free subquotients. Let $K_i$ be the kernel of the map $q_i \co Y \to (Y_p / Y_p \cap L_{i-1})
\boxtimes_A W^*(p)$. We claim that $K_i$ is dual Weyl filtered, and has shorter length than $Y$. After all, $L_j \subset K_i$ unless $q_j = p$, and when $q_j = p$, $L_j \subset K_i$
precisely when $j < i$. Moreover, $L_i \cap K_i = L_{i-1} \cap K_i$. So we have \[(L_{j}\cap K_{i})/(L_{j-1}\cap K_{i})\cong \begin{cases} L_{j}/L_{j-1} &q_j\neq p\textrm{ or } j<i\\ 0 &
\textrm{else}. \end{cases} \] Thus $L_j \cap K_i$ is a dual Weyl filtration of $K_i$.

Suppose that $p \ge n$. Then the induced map $X \to Y_p \boxtimes_A W^*(p)$ is zero, so $X \subset K_r$. This contradicts the minimality of $Y$, so we can assume $p < n$.


Let $q$ denote the map $Y \to Y_p \boxtimes_A W^*(p)$, and let $K$ be its kernel. Then $K$ has a dual Weyl filtration, and has maximal weight $<p$. Let $\pi_Y$ denote the surjective map
$\pi \co Y_p\boxtimes_A V^{\ot p} \to Y_p\boxtimes_A W^*(p)$. By Lemma \ref{lemma:stdcostdhom}, $\Ext^1(Y_p\boxtimes_A V^{\ot p} ,K)=0$. Thus, we can lift $\pi_Y$ to a map $\tilde{\pi}_Y
\co Y_p\boxtimes_A V^{\ot p} \to Y$. The image of $\tilde{\pi}_Y$ together with $K$ spans $Y$. By induction, $K$ is spanned by the images of maps from $V^{\ot k}$ with $k<p<n$. Thus, we
have proved that the images of maps from $V^{\ot k}$ for $k<n$ span $M[<n]$.

Now let $\phi\colon N\to M[<n]$ be a map from a Weyl filtered module. Let $p$ be the maximal weight space in $N$ on which this map is non-zero; the image of our map lies in $M[\le p]$. We
will prove the result by induction on $p$. Consider the projection $q\colon M[\leq p]\to M[\leq p]_p\boxtimes_A W^*(p)$. As above, we also have a map \[\pi_{M[\le p]} \colon M[\leq
p]_p\boxtimes_A V^{\ot p} \to M[\leq p]_p\boxtimes_A W^*(p),\] and the same Ext-vanishing argument shows that we can define a lift $\tilde{\pi}_{M[\le p]} \colon M[\leq p]_p\boxtimes_A V^{\ot p} \to
M[\leq p]$ such that $\pi =q \circ \tilde{\pi}$ as shown in \eqref{eq:lift-diagram} below.

The composition $q\circ \phi $ is a map from a Weyl filtered module to $M[\leq p]_p\boxtimes_A W(p)^*$. Note that the kernel $J$ of the map $\pi$ is also dual Weyl filtered so
$\Ext^1(N,J)=0$. Thus, the map $q\circ \phi $ can be lifted to a map $\eta\colon N\to M[\leq p]_p\boxtimes_A V^{\ot p}$ which satisfies $q\circ \phi=\pi \circ \eta$.
\begin{equation} \tikz[->,very thick,baseline]{
\matrix[row sep=15mm,column sep=20mm,ampersand replacement=\&]{
	\node (a) {$N$};
	\& \node (c) {$M[\leq p]_p\boxtimes_A V^{\ot p}$};
	\\ \node (b) {$M[\leq p]$};
	\& \node (d) {$M[\leq p]_p\boxtimes_A W^*(p)$};
	\\ };
	\draw (a) --node[left,midway]{$\phi$} (b);
	\draw[dashed] (a) -- node[midway,above]{$\eta$} (c);
	\draw (c) --node[right,midway]{$\pi$} (d);
	\draw (b)-- node[midway,below]{$q$} (d);
	\draw[dashed] (c)-- node[midway,below right]{$\tilde{\pi}$} (b);
}\label{eq:lift-diagram} \end{equation}
Thus, the map $\phi'=\phi- \pi \circ \eta $ lands in $M[<p]$. By induction, $\phi'$ is a sum of maps that factor through $V^{\ot k}$ for $k<p$, so this completes the proof.
\end{proof}
Note that the ideal $I_{<n}$ is precisely the set of maps $V^{\ot
  n}\to V^{\ot n}$ which factors through $V^{\ot k}$ for $k<n$.  Thus,
applied to $V^{\ot n}$, this shows that the image of $I_{<n}$ is
precisely $V^{\ot n}[<n]$.

\begin{proof}[Proof of Theorem \ref{thm:JW-exist}] If the Jones-Wenzl projector $J_n$ exists, then we can consider its action on $T = (V^{\ot n})'$. The idempotent $1 - J_n$ acts as the identity on
$T[<n]$, and since it lies in $I_{<n}$, its image is also contained in $T[<n]$. Thus, $1 - J_n$ is projection to $T[<n]$. Let $X$ denote its complement, the image of $J_n$. Since
there is a short exact sequence \[0\to T[<n]\to T \to W^*(n)'\to 0,\] we have that $X\cong W^*(n)'$. On the other hand, since $J_n$ acts by the identity on the $n$-weight space, the
induced map $W(n)' \to X$ must be an isomorphism as well. Thus, the natural map $W(n)' \to T \to W^*(n)'$ must be an isomorphism.

Under this map $F^{(k)}v_n$ is sent to the map
\[F^{(k)}w_n = w_n\circ (-1)^kq^{k(k-1)/2}K^{-k}E^{(k)}\colon W(n)'\to R',\]
since $S(F^{(k)})=-q^{k(k-1)/2}K^{-k}E^{(k)}$. Calculating, this is given by \[(-1)^kq^{k(k-1)/2-nk}\qbinom nk w_{n-2k}.\]
This spans the $n-2k$ weight space (and thus the map is an isomorphism) if and only if $\qbinom nk$ is a unit in $R'$.

On the other hand, assume that $\qbinom nk$ is always a unit. In this case, the induced map $W(n)'\to T \to W^*(n)'$ is an isomorphism by the same calculation. Thus, the image of $W(n)'$
is a complementary submodule to $T[<n]$ inside $T$. The projection to $T[<n]$ must be in the ideal $I_{<n}$, so the complementary projection to $W(n)'$ must be the image of a
Jones-Wenzl projector under $\rho'$. Thus $J_n$ must exist. \end{proof}


\bibliographystyle{plain}
\bibliography{everyone}{}

\begin{thebibliography}{10}

\bibitem{beilinson1981localisation}
Alexander Beilinson and Joseph Bernstein.
\newblock Localisation de g-modules.
\newblock {\em CR Acad. Sci. Paris}, 292(1):15--18, 1981.

\bibitem{bourbaki1968elements}
Nicolas Bourbaki.
\newblock {\em Elements de mathematique, fasc. 34: groupes et alg{\`e}bres de
  Lie, chap. 4, 5 et 6}.
\newblock Hermann, 1968.

\bibitem{brylinski1981kazhdan}
Jean-Luc Brylinski and Masaki Kashiwara.
\newblock Kazhdan-{L}usztig conjecture and holonomic systems.
\newblock {\em Inventiones mathematicae}, 64(3):387--410, 1981.

\bibitem{DPS}
Jie Du, Brian Parshall, and Leonard Scott.
\newblock Quantum {W}eyl reciprocity and tilting modules.
\newblock {\em Comm. Math. Phys.}, 195(2):321--352, 1998.

\bibitem{DyerUniversal}
Matthew Dyer.
\newblock On some generalisations of the {K}azhdan-{L}usztig polynomials for
  ``universal" {C}oxeter systems.
\newblock {\em Journal of Algebra}, 116(2):353--371, 1988.

\bibitem{EDihedral}
Ben Elias.
\newblock The two-color {S}oergel calculus.
\newblock {\em Compos. Math.}, 152(2):327--398, 2016.

\bibitem{EQAGS}
Ben Elias.
\newblock Quantum {S}atake in type {A}: part {I}.
\newblock {\em J. Comb. Algebra}, 1(1):63--125, 2017.
\newblock arXiv:1403.5570.

\bibitem{EWHodge}
Ben Elias and Geordie Williamson.
\newblock The {H}odge theory of {S}oergel bimodules.
\newblock Preprint.
\newblock arXiv:1212.0791.

\bibitem{EWGR4SB}
Ben Elias and Geordie Williamson.
\newblock Soergel calculus.
\newblock Preprint.
\newblock arXiv:1309.0865.

\bibitem{fiebig2008combinatorics}
Peter Fiebig.
\newblock The combinatorics of {C}oxeter categories.
\newblock {\em Transactions of the American Mathematical Society},
  360(8):4211--4233, 2008.

\bibitem{GooWen}
Frederick~M Goodman and Hans Wenzl.
\newblock Ideals in the {T}emperley {L}ieb category.
\newblock 2002.
\newblock arXiv:math/0206301.

\bibitem{Humphreys}
James~E. Humphreys.
\newblock {\em Reflection groups and {C}oxeter groups}, volume~29 of {\em
  Cambridge Studies in Advanced Mathematics}.
\newblock Cambridge University Press, Cambridge, 1990.

\bibitem{JMW}
Daniel Juteau, Carl Mautner, and Geordie Williamson.
\newblock Perverse sheaves and modular representation theory.
\newblock In {\em Geometric methods in representation theory II}, volume~25 of
  {\em S\'eminaires et {C}ongr\`es}, pages 313--350. Soc. Math. France, Paris,
  2010.

\bibitem{KaLu1}
David Kazhdan and George Lusztig.
\newblock Representations of {C}oxeter groups and {H}ecke algebras.
\newblock volume~53, pages 165--184. 1979.

\bibitem{LauSL2}
Aaron~D. Lauda.
\newblock A categorification of quantum {${\rm sl}(2)$}.
\newblock {\em Adv. Math.}, 225(6):3327--3424, 2010.

\bibitem{LauDiagrams}
Aaron~D. Lauda.
\newblock An introduction to diagrammatic algebra and categorified quantum
  {$\rm{sl}_2$}.
\newblock {\em Bull. Inst. Math. Acad. Sin. (N.S.)}, 7(2):165--270, 2012.

\bibitem{LenzingLectureNotes}
Helmut Lenzing.
\newblock Hereditary categories.
\newblock
  http://webusers.imj-prg.fr/~bernhard.keller/ictp2006/lecturenotes/lenzing1.pdf.

\bibitem{LibRR}
Nicolas Libedinsky.
\newblock Sur la cat\'egorie des bimodules de {S}oergel.
\newblock {\em J. Algebra}, 320(7):2675--2694, 2008.

\bibitem{LibLIF}
Nicolas Libedinsky.
\newblock Light leaves and {L}usztig's conjecture.
\newblock {\em Adv. Math.}, 280:772--807, 2015.

\bibitem{Mor}
Scott Morrison.
\newblock A formula for the {J}ones-{W}enzl projections.
\newblock preprint, available at
  http://tqft.net/math/JonesWenzlProjections.pdf.

\bibitem{Soe2}
Wolfgang Soergel.
\newblock The combinatorics of {H}arish-{C}handra bimodules.
\newblock {\em J. Reine Angew. Math.}, 429:49--74, 1992.

\bibitem{Soe4}
Wolfgang Soergel.
\newblock On the relation between intersection cohomology and representation
  theory in positive characteristic.
\newblock {\em J. Pure Appl. Algebra}, 152(1-3):311--335, 2000.
\newblock Commutative algebra, homological algebra and representation theory
  (Catania/Genoa/Rome, 1998).

\bibitem{Soe5}
Wolfgang Soergel.
\newblock Kazhdan-{L}usztig-{P}olynome und unzerlegbare {B}imoduln \"uber
  {P}olynomringen.
\newblock {\em J. Inst. Math. Jussieu}, 6(3):501--525, 2007.

\bibitem{TemLie}
H.~N.~V. Temperley and E.~H. Lieb.
\newblock Relations between the ``percolation'' and ``colouring'' problem and
  other graph-theoretical problems associated with regular planar lattices:
  some exact results for the ``percolation'' problem.
\newblock {\em Proc. Roy. Soc. London Ser. A}, 322(1549):251--280, 1971.

\bibitem{Wenzl}
Hans Wenzl.
\newblock On sequences of projections.
\newblock {\em C. R. Math. Rep. Acad. Sci. Canada}, 9(1):5--9, 1987.

\bibitem{WestburyTL}
B.~W. Westbury.
\newblock The representation theory of the {T}emperley-{L}ieb algebras.
\newblock {\em Math. Z.}, 219(4):539--565, 1995.

\bibitem{WestburyCellular}
Bruce~W. Westbury.
\newblock Invariant tensors and cellular categories.
\newblock {\em J. Algebra}, 321(11):3563--3567, 2009.

\bibitem{WilCounterexample}
Geordie Williamson.
\newblock {S}chubert calculus and torsion.
\newblock Preprint, arXiv:1309.5055.

\bibitem{WilSSB}
Geordie Williamson.
\newblock Singular {S}oergel bimodules.
\newblock {\em Int. Math. Res. Not. IMRN}, (20):4555--4632, 2011.

\end{thebibliography}

\vspace{0.1in}
 
\noindent
{\textsl \small Ben Elias, Department of Mathematics, University of Oregon, Eugene, OR, USA}

\noindent 
{\tt \small email: belias@uoregon.edu}

The first author was supported by NSF grant DMS-1103862.

\vspace{0.1in}
 
\noindent
{\textsl \small Nicolas Libedinsky, Department of Mathematics, Universidad de Chile, Santiago, Chile}

\noindent 
{\tt \small email: nlibedinsky@gmail.com}

The second author was supported by Fondecyt iniciacion 11121118.

\end{document}